\documentclass[10pt]{scrartcl}%                 

\usepackage{graphicx}
\RequirePackage[utf8]{inputenc}
\RequirePackage{amsfonts,amsmath,amssymb}
\RequirePackage{subfig} %replacement for subfigure
\RequirePackage{algorithm2e}

\RequirePackage[colorlinks=true,linkcolor=blue,citecolor=blue,urlcolor=blue]{hyperref}
\RequirePackage{amsthm}
\RequirePackage[scale=0.75]{geometry}

\newtheorem{theorem}{Theorem}
\newtheorem{lemma}[theorem]{Lemma}
\newtheorem{corollary}[theorem]{Corollary}

\RequirePackage[colorlinks=true,linkcolor=blue,citecolor=blue,urlcolor=blue]{hyperref}
\RequirePackage[english]{babel}
\theoremstyle{definition}
\newtheorem{hyp}[theorem]{Assumption}
\newtheorem{remark}[theorem]{Remark}
\newtheorem{definition}[theorem]{Definition}

\newcommand*{\tnoteref}[1]{\protect\footnotemark}
\newcommand*{\tnotetext}[2][]{\footnotetext{#2}}

\newcommand*{\ead}[1]{\stepcounter{footnote}\footnotetext[\thefootnote]{#1}}
\newcommand*{\cortext}[2][]{}
\newcommand*{\corref}[1]{}
\newcommand*{\address}[2][]{{\centering$^{#1}$\itshape\small#2\par}}
\newcommand*{\artauthor}[1]{\begin{center}\large#1\end{center}}
\renewcommand*{\author}[2][]{}
\newcommand*{\MSC}{\paragraph{2000 MSC:}}
\newcommand*{\sep}{,\, }
\renewcommand*{\title}[1]{\begin{center}\sffamily\Large#1\end{center}}
\date{}
\renewenvironment{abstract}{\vspace{2em}\hrule\noindent\paragraph{Abstract:}}{}
\newenvironment{keyword}{\noindent \paragraph{Keywords:}}{}
\newenvironment{frontmatter}{}{\vspace{1em}\hrule}
\renewenvironment{proof}{\noindent\textit{Proof.}}{}
\newcommand*{\R}{\mathbb{R}}
\newcommand*{\C}{\mathbb{C}}
\newcommand*{\etoile}{^{\star}}
\newcommand*{\range}[1]{\mathcal{R}\left(#1\right)}
\newcommand*{\I}{^{\infty}}
\newcommand*{\rest}[1]{\displaystyle\left|_{ #1}\right.}
\newcommand*{\dn}{\partial_{\nu}}
\newcommand*{\grad}{\nabla}
\newcommand*{\rond}{\circ}
\newcommand*{\petito}[1]{\ensuremath{\mathop{}\mathopen{}{\scriptstyle\mathcal{O}}\mathopen{}\left(#1\right)}}
\newcommand*{\abs}[1]{\left\lvert#1\right\rvert}
\newcommand*{\norme}[1]{\left\lVert#1\right\rVert}
\newcommand*{\ps}[2]{\left\langle#1,\,#2\right\rangle}
\newcommand*\mat[1]{\begin{pmatrix}#1\end{pmatrix}}
\newcommand*{\pardef}{\mathrel:=}
\newcommand*{\inv}{^{-1}}
\newcommand*{\re}{{\mathrm{Re}}\:}
\newcommand*{\im}{{\mathrm{Im}}\:}

\renewcommand\appendix{\par
\setcounter{section}{0}%
\setcounter{subsection}{0}%
\setcounter{table}{0}
\setcounter{figure}{0}
\setcounter{equation}{0}
\gdef\thetable{\Alph{table}}
\gdef\thefigure{\Alph{figure}}
\gdef\theequation{\Alph{section}-\arabic{equation}}
\section*{Appendix}\small
\gdef\thesection{\Alph{section}}
\setcounter{section}{1}}

\begin{document}

\begin{frontmatter}

\title{An optimization approach for the localization of defects in an inhomogeneous medium from acoustic far-field measurements at a fixed frequency \tnoteref{frae}}
\tnotetext[frae]{Support for some of the authors of this work was provided by the FRAE (Fondation de Recherche pour l'A\'eronautique et l'Espace, \texttt{http://www.fnrae.org/}), research project IPPON.}

\artauthor{Yann Grisel$^{a,}$\footnotemark[2], Jérémie Fourbil$^{b}$ and Vincent Mouysset$^{b,}$\footnotemark[3]}
\author[a]{Yann Grisel\corref{cor1}}
\ead{yann.grisel@iut-tlse3.fr}
\author[b]{Jérémie Fourbil}
\author[b]{Vincent Mouysset\corref{cor2}}
\ead{mouysset@onera.fr}

\cortext[cor1]{Corresponding author}
\cortext[cor2]{Principal corresponding author}

\address[a]{Institut de Mathématiques de Toulouse, 31062 Toulouse, France}
\address[b]{Onera - The French Aerospace Lab, 31055 Toulouse, France}

\begin{abstract}
We are interested in the localization of defects in non-absorbing inhomogeneous media with far-field measurements generated by plane waves.
In localization problems, most so-called sampling methods are based on a characterization involving point-sources and the range of some implicitly defined operator.
We present here a way to deal with this implicit operator by the means of an optimization approach in the lines of the well-known inf criterion for the factorization method.
\end{abstract}


\end{frontmatter}

\section{Introduction}
 
\begin{figure}[htpb]
\centering
\includegraphics[height=3cm]{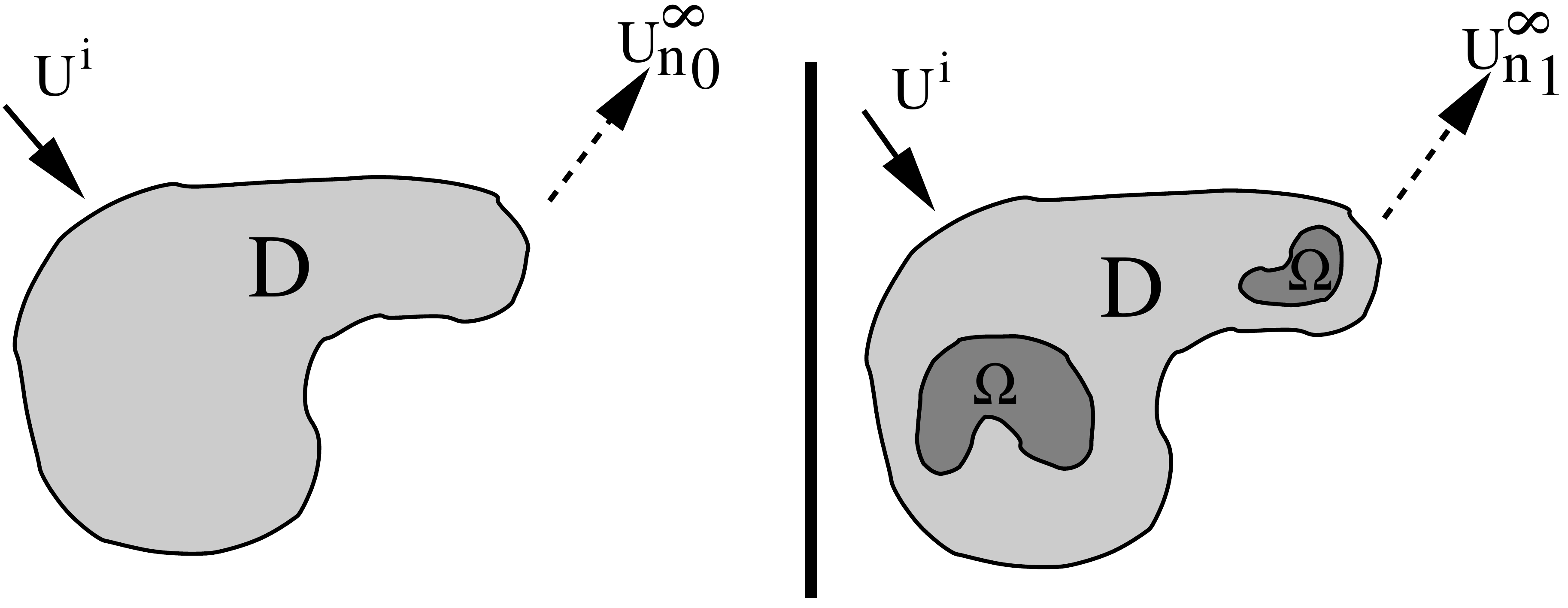}
\caption{Reference setting (left) and actual setting (right).}
\label{fig:schema.inclusion}
\end{figure}

%~\cite{art.potthast.98,art.potthast.04}
%~\cite{art.colton.96,art.kirsch.98}
%~\cite{art.kirsch.00}

%This brings us back to the localization result~\cite{art.grisel.12} issued from the classical factorization method~\cite{art.kirsch.02} but, this time, coming from a minimization approach.

We consider an inverse scattering problem consisting in shape reconstruction from physical measurements.
These problems are generally non-linear and ill-posed.
More specifically, we address the problem of reconstructing the support of a perturbation in a given inhomogeneous background medium from acoustic far-field measurements generated with plane waves. 
It may indeed happen that, in some places, the actual index is different from the reference value, as seen in Figure~\ref{fig:schema.inclusion}.
This could happen for instance from a deterioration or an incorrect estimation of the actual index. 
We then say that there is a defect at any point where the reference index is different from the actual index. 

A wide range of methods achieve the localization of obstacles by a \emph{sampling} approach: the points of the unknown domain are characterized by a binary test that has to be applied to the whole space.
For most of them, the first formulation of this pointwise test is to check if some well chosen test-function is in the range of an implicitly defined operator.
See~\cite{art.colton.03,art.potthast.06} and references therein for a topical review.
A natural way to proceed is then to connect the range of the implicit operator to the range of an operator explicitly defined from the actually available measurements.
This is the principle of the \emph{linear sampling} method~\cite{art.colton.96,art.colton.97,art.collino.2002} or of the \emph{factorization} method~\cite{art.kirsch.98,art.kirsch.02,art.arens.04}. %:
Yet, when looking for perturbations in non-homogeneous background media, it is only recently that a factorization method has been proposed to reconstruct the shape of defects~\cite{art.nachman.07,art.grisel.12}.

However, we investigate in this paper an optimization approach in the lines of the \emph{inf} criterion~\cite{art.kirsch.00} to deal with the implicit operator's range.
We show that this leads to a characterization of the defects as the support of the following function: 
\[\mathcal M_{W}(z) \pardef \inf\left\{ f_{W}(\Psi),\ \Psi \in L^2(S^{d-1}) \text{ and } {\ps{\Psi}{\overline{u_{n_0}(\cdot,z)}}}_{L^2(S^{d-1})} = 1\right\},\]
%\[ z \in \Omega \iff \mathcal M_{W}(z) >0,\]
where $u_{n_0}$ is the total field for the reference index and the form $f_W$ is given by 
\[f_{W}(\Psi) \pardef \abs{\ps{ W\Psi}{\Psi}_{L^2(S^{d-1})}}^{\frac{1}{2}},\]
with a measurements operator $W$ explicitly built from the available data.
Throughout this paper, $\ps{\cdot}{\cdot}_X$ stands for the usual hermitian inner product in the Hilbert space $H$.
%We also show how this characterization can be linked to the range identity of the aforementioned factorization method's $F_\#$ variant.

%\[\mathcal M_{W_\#}(z) >0 \iff \overline{u_{n_0}(\cdot,z)} \in \range{W_\#^\frac{1}{2}}.\]

This paper is structured as follows.
In section~\ref{sec:presentation}, the mathematical setting is specified and the implicit localization of the defects is recalled.
This localization is then expressed as a binary pointwise test involving an constrained optimization problem in section~\ref{sec:optimisation}.
%Some technical justifications will be set apart to section~\ref{sec:coercivite}.
Finally, in section~\ref{sec:iteratives}, we investigate numerical methods for solving this optimization problem.
%Finally, we rewrite the minimization problem in section~\ref{sec:range.identity} to obtain the range identity used in the classical factorization method.
We end by some conclusions.

\section{Formal localization}\label{sec:presentation}

If we consider time-harmonic acoustic waves with a fixed wave number \(k\), the spatial part of the wave equation is modeled by the Helmholtz equation~\cite{book.colton.1}. 
Inhomogeneous media are then represented by an acoustic refraction index, denoted by \(n \in L\I(\R^{d})\), and so the total field, denoted by \(u_n\),
is assumed to satisfy 
\begin{equation}\label{eq:Helmholtz}
\Delta u_n +k^2n(x) u_n = 0, \quad x \in \R^d,
\end{equation}
where \(d\) is the problem's dimension (\(d = 2\) or 3).
We consider compactly supported inhomogeneities and denote by \(D\) the support of \(n(x)-1\). 
Also, we denote an incoming wave satisfying~(\ref{eq:Helmholtz}) with \(n=1\)  by \(u^i \in L^2_{loc}(\R^{d})\).
The total field is then the sum of this incoming wave and the wave scattered by the inhomogeneous medium, denoted by  \(u^s \in L^2_{loc}(\R^{d})\):
\begin{equation}\label{eq:u_n}
    u_n \pardef u^s+u^i,
\end{equation}
where the scattered wave is assumed to satisfy the Sommerfeld radiation condition
\begin{equation}\label{eq:Sommerfeld}
    \partial_{r}u^s=iku^s+\petito{\abs{x}^{-\frac{d-1}{2}}}.
\end{equation}
Then, the linear system~(\ref{eq:Helmholtz})-(\ref{eq:Sommerfeld}) defines \(u_n\) uniquely from \(u^i\), and it is known to be invertible in \(L^2(D)\). 
Thus, let us denote the corresponding automorphism by
\[\begin{array}{rclcl}
    \mathcal{T}_n &: &L^2(D) &\to &L^2(D),\\
         &&u^i &\mapsto &u_n.
\end{array}\]

Besides, the outgoing part of a wave has an asymptotic behaviour called the far field pattern, 
denoted by \(u_n\I \in L^2(S^{d-1})\), see figure~\ref{fig:schema.scattering}, and given by the Atkinson expansion~\cite{conf.venkov.08}
\[
    u_n(x) \pardef u^i(x)+\gamma \frac{e^{ik\abs{x}}}{\abs{x}^{\frac{d-1}{2}}}u_n\I(\hat x)+\petito{\abs{x}^{-\frac{d-1}{2}}},\quad \hat x \pardef \frac{x}{\abs{x}} \in S^{d-1},
\]
%where $\Gamma_m$ denotes the set of measurement directions as a subset of the unit sphere \(S^{d-1}\) (see figure~\ref{fig:schema.scattering}) and 
where \(\gamma\) depends only on the dimension and is defined by
\[\gamma \pardef 
    \begin{cases}  \frac{ e^{i\pi /4} }{ \sqrt{8\pi k} } & \text{if } d=2,  \\  \frac{1}{4\pi} & \text{if } d=3.  \end{cases}
\]

\begin{figure}[tpb]
\centering
\includegraphics[height=3cm]{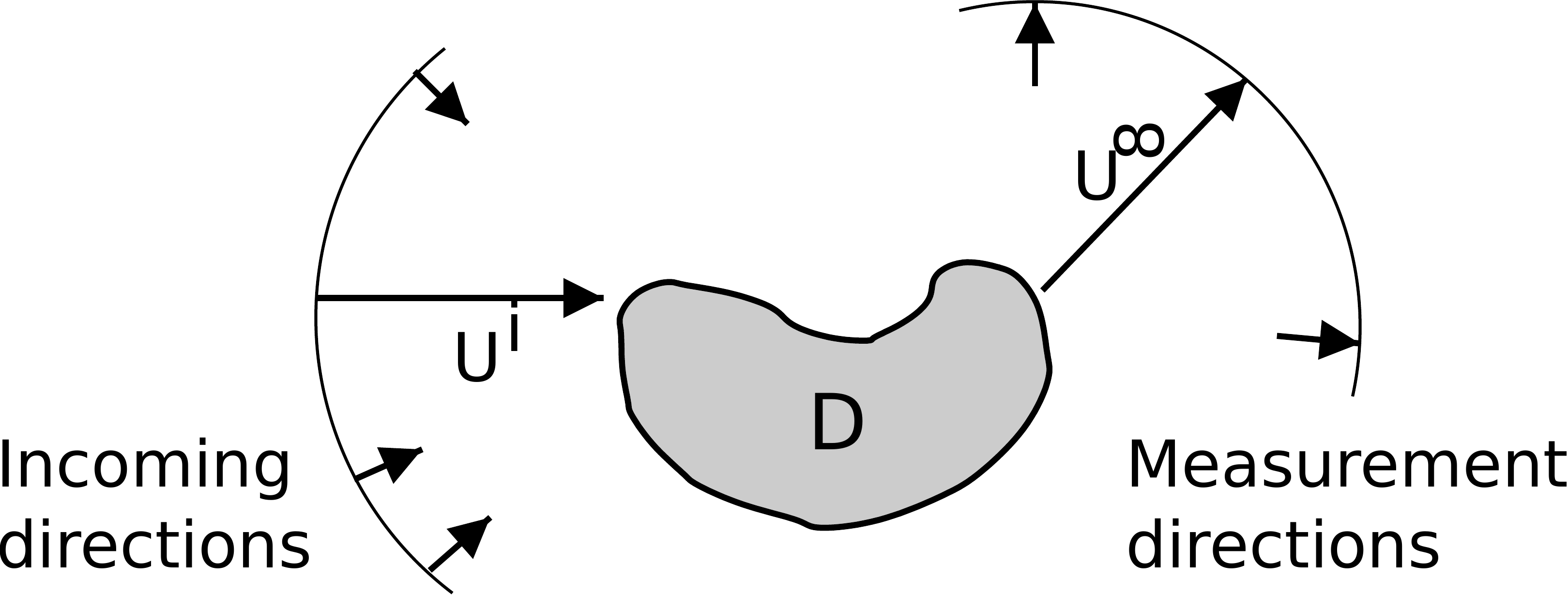}
\caption{General setting and notations.}
\label{fig:schema.scattering}
\end{figure}

Furthermore, for practical reasons, we will mainly consider scattered waves having a plane-wave source.
These plane-waves are defined by \[u^i(\hat\theta,x) \pardef exp({ik\hat\theta \cdot x}),\] where \(\hat\theta\) is a unitary vector in \(S^{d-1}\) as depicted on figure~\ref{fig:schema.scattering}.
Then, let us denote the total field at the point \(x \in \R^d\) and with a plane-wave source of incoming direction \(\hat \theta\), by
\[{u_n(\hat\theta,x)} \pardef \mathcal{T}_n(u^i(\hat\theta,\cdot))(x).\]
The corresponding far-field pattern in the measurement direction $\hat x \in S^{d-1}$ will be denoted by $u_n\I(\hat \theta, \hat x)$.
Lastly, these measurements will be used in the form of the classical far-field operator \(F_n : L^2(S^{d-1}) \to L^2(S^{d-1})\), defined by
\[F_ng (\hat x) \pardef \ps{g}{\overline{u_n\I(\cdot,\hat x)}}_{L^2(S^{d-1})}.\]

We want to reconstruct the shape of defects in a reference medium whose index is denoted by \(n_0 \in L\I(D)\). 
Let then \(n_1 \in L\I(D)\) denote the actual index, altered by the presence of these defects.
So,  denote the support of the difference between the two indices (see figure~\ref{fig:schema.inclusion}) by \[\Omega \pardef \mathrm{support}(n_1-n_0).\] 
Thus, the goal is to reconstruct the domain \(\Omega\), through its associated characteristic function denoted $\mathbf{1}_\Omega$, from the reference index \(n_0\) and far-field measurements \(u_{n_1}\I\).
Yet, it has been shown that the unknown domain can be characterized by a set of test functions and the range of some implicitly defined operator.

\begin{theorem}[Implicit domain characterization]\cite[Theorem 3.2]{art.grisel.12}\label{thm:omega.1}
    Let us define the operator \(C : L^2(D) \to L^2(S^{d-1})\) by
    \[Cf(\hat x) = \langle f,u_{n_0}(\hat x,\cdot)\rangle_{L^2(D)}.\]
    Then, for each \(z \in \R^d\), we have
    \begin{equation}\label{eq:omega.1}
    z \in \Omega \iff \overline{u_{n_0}(\cdot,z)} \in \range{C\mathbf{1}_{\Omega}},
    \end{equation}
    where $\mathbf{1}_{\Omega} : L^2(\Omega) \to L^2(D)$ is the restriction to $\Omega$ given by
    $$\mathbf{1}_{\Omega}f(\hat x) = \begin{cases}0,\quad x \notin \Omega\\f(x), \quad x \in \Omega\end{cases}.$$
\end{theorem}

\section{Explicit identification of the defects}\label{sec:optimisation}

This section gives an explicit formulation of characterization~(\ref{eq:omega.1}). 
To do so, we proceed in three steps. 
First, we introduce a practical result formulating the belonging to the range of any given (bounded) operator $L$ as an \textit{inf} criterion based on any function related to $L$ through specific assumptions. 
Then, we construct such a well-suited function for the operator we are interested in, namely $C\mathbf{1}_{\Omega}$. 
Finally, we can give the explicit characterization of the defect~$\Omega$.

\subsection{Formulation of a bounded operator's range through an \textit{inf} criterion}

To deal with the range of $C\mathbf{1}_{\Omega}$, depending on the unknown domain $\Omega$, let us recall the following characterization of an operator's range.

\begin{lemma}[Range characterization]\cite[Lemma 2.1]{art.nachman.07}\label{lem:image}
    Let \(L : H_1 \to H_2\) be a bounded operator between two Hilbert spaces $H_1$ and $H_2$, and let \(\varphi \in H_2\).
    Then, noting $H_2\etoile$ the dual of $H_2$ and identifying $H_1$ with its dual, \(\varphi \in \range{L}\) if and only if there exists \(c>0\) such that for all \(\Psi \in H_2\etoile\) 
    \begin{equation}\label{eq:inegalite.1}
    \abs{\ps{\Psi}{\varphi}} \leqslant c\norme{L\etoile \Psi}.
    \end{equation}
\end{lemma}
Hence, any form comparable to $\norme{L\etoile (\cdot)}$ can be used to characterize the range of the operator $L$.

\begin{corollary}\label{cor:image.1}
With $L$ as in lemma~\ref{lem:image}, let \(f : H_2\etoile \to \R\) be comparable to \(L\etoile\) in the sense that there exists \(c_1>0\) and \(c_2>0\) such that 
\begin{equation}\label{eq:equivalence}
	c_1\norme{L\etoile\Psi} \leqslant f(\Psi) \leqslant c_2\norme{L\etoile\Psi}, \quad \Psi \in H_2\etoile.
\end{equation}
    Then, \(\varphi \in \range{L}\) if and only if there exists \(c_3>0\) such that for all \(\Psi \in H_2\etoile\) 
    \begin{equation}\label{eq:inegalite.2}
    \abs{\ps{\Psi}{\varphi}} \leqslant c_3f(\Psi).
    \end{equation}
\end{corollary}
\begin{proof}
The result~(\ref{eq:inegalite.2}) is a straightforward combination of the characterization~(\ref{eq:inegalite.1}) and~(\ref{eq:equivalence}).
\qed\end{proof}
Finally, from corollary~\ref{cor:image.1}, we deduce the following characterization based on an \textit{inf} criterion.
\begin{corollary}\label{cor:image.2}
    With $f$ as in corollary~\ref{cor:image.1}, then \(\varphi \in \range{L}\) if and only if \[0< \inf\left\{ f(\Psi),\ \Psi \in H_2\etoile\text{ and }{\ps{\Psi}{\varphi}}=1 \right\}.\]
\end{corollary}
\begin{proof}
First, if \(\varphi \not\in \range{L}\), from corollary~\ref{cor:image.1}, for each $c>0$ there exists $\Psi_c \in H_2\etoile$ such that \(\abs{\ps{\Psi_c}{\varphi}} > cf(\Psi_c).\)
Since \(f(\Psi_c) \geqslant 0\) by~(\ref{eq:equivalence}), then \({\ps{\Psi_c}{\varphi}} \not= 0\) and we can define \(\psi_c \pardef \Psi_c /{\ps{\Psi_c}{\varphi}}\).
From~(\ref{eq:equivalence}), it follows 
\[1 = \abs{\ps{\Psi_c}{\varphi}}\frac{1}{\abs{\ps{\Psi_c}{\varphi}}} > cf(\Psi_c)\frac{1}{\abs{\ps{\Psi_c}{\varphi}}} \geqslant cf(\psi_c)\frac{c_1}{c_2}.\]
So, there is a set of functions $\psi_c \in  H_2\etoile$, satisfying ${\ps{\psi_c}{\varphi}}=1$, such that $f(\psi_c) \to 0$ when $c \to \infty$.

Next, let \(\varphi \in H_2\setminus\{0\}\) be in the range of the operator $L$ and let \(\Psi \in H_2\etoile\) satisfy \(\ps{\Psi}{\varphi} = 1\).
Thus, from corollary~\ref{cor:image.1}, there exists \(c_3>0\) such that \( 1 \leqslant c_3f(\Psi)\) and the infimum  is not vanishing.
Finally, if $\varphi =0$, which is always in the range of $L$, the infimum is evaluated over an empty set and will then conventionally be given the value~$+\infty$.
\qed\end{proof}

\subsection{Construction of a well-suited objective function}

 As shown in corollary~\ref{cor:image.2}, the range of a linear bounded operator $L$ can be formulated as a usual constrained optimization problem without the exact knowledge of $L$.
Indeed, it suffices to find any form $f$ satisfying~(\ref{eq:equivalence}).
Hence, to get an explicit characterization of the domain $\Omega$ from theorem~\ref{thm:omega.1}, we look for a form $f$ satisfying 
\begin{equation} \label{eq:encadrement.1}
	c_1\norme{\mathbf{1}_{\Omega}C\etoile\Psi} \leqslant f(\Psi) \leqslant c_2\norme{\mathbf{1}_{\Omega}C\etoile\Psi}, \quad \Psi \in L^2(S^{d-1}).
\end{equation}
To achieve this, following~\cite{art.kirsch.00}, we consider the form \(f_{W} : L^2(S^{d-1}) \to \R_+\)  defined by
\begin{equation}\label{def:fW.1}
    f_{W}(\Psi) \pardef \abs{\ps{ W\Psi}{\Psi}_{L^2(S^{d-1})}}^{\frac{1}{2}},
\end{equation}
where $W$ is a measurements operator explicitly built from the available data.

A first guess for the operator $W$ would be to consider $(F_{n_1}-F_{n_0})$: the difference between the far-field operators corresponding respectively to the reference index $n_0$ and the actual index $n_1$.
But it has been shown in~\cite{art.grisel.12} that this subtraction does not yield some crucial factorization.
Thus, we have to restrain ourselves to full bi-static data ($u_{n_0}\I$ and $u_{n_1}\I$ are known over \(S^{d-1} \times S^{d-1}\)) and non-absorbing media ($n_0(x)$ and $n_1(x) \in \R$) so we can consider the operator $ W : L^2(S^{d-1}) \to L^2(S^{d-1})$ defined by
\[W \pardef \left(\mathrm{id}+2ik \abs{\gamma}^2 F_{n_0}\right) \left(F_{n_1}-F_{n_0}\right).\] 

Under these assumptions,  it has been shown that this measurements operator $W$ has a factorization of the form (see \cite[Corollary 4.7, Corollary 4.3]{art.grisel.12})
\begin{equation}\label{eq:factorisation.1}
	W=CAC\etoile,
\end{equation}
where the operator $A$ is an automorphism on $L^2(\Omega)$ defined by
\begin{equation}\label{def:A}
	A \pardef \mathbf{1}_\Omega k^2(n_1-n_0)\mathcal T_{n_1} \mathcal T_{n_0}\inv \mathbf{1}_\Omega.
\end{equation}
So, we have
\begin{equation}\label{eq:factorisation.2}
	f_W(\Psi) = \abs{\ps{CAC\etoile\Psi}{\Psi}_{L^2(S^{d-1})}}^{\frac{1}{2}} = \abs{\ps{AC\etoile\Psi}{C\etoile\Psi}_{L^2(\Omega)}}^{\frac{1}{2}},
\end{equation}
and as such, the inequalities~(\ref{eq:encadrement.1}) are then the continuity and the coercivity of the operator $A$, at least on the range of \(C\etoile\).
We are now going to show that this coercivity is related to the contrast between the reference index \(n_0\) and the actual values of \(n_1\), i.e. the defects should be clearly distinguished from the background.
We thus make the following geometrical assumption.

\begin{hyp}\label{hyp:coercivite}
    Assume that \(n_{0}\) and \(n_{1}\) are real valued and that either \((n_1-n_0)\) or \((n_0-n_1)\) is locally bounded from below:
    \begin{itemize}
    \item for any compact subset \(\omega\) included in \(\Omega\), there exists \(c>0\) such that \((n_1(z)-n_0(z)) \geqslant c\) for almost all \(z \in \omega\),
    \end{itemize}
    or
    \begin{itemize}
    \item for any compact subset \(\omega\) included in \(\Omega\), there exists \(c>0\) such that \((n_0(z)-n_1(z)) \geqslant c\) for almost all \(z \in \omega\).
    \end{itemize}
\end{hyp}

Moreover, for a fixed geometry of defects, some wave numbers $k$ may produce resonances that cancel the outgoing wave.
Indeed, the operator  mapping an incoming wave to the corresponding far-field pattern is not one-to-one in the case of inhomogeneous media.
This corresponds to the so-called interior transmission eigenvalues arise~\cite{art.kirsch.99}.

\begin{definition}\label{def:transmission}
We call \(k\) an \emph{interior transmission eigenvalue for the pair of indices \((n_0,n_1)\)} if there exists a non-vanishing (source,solution) pair, denoted by  \((h,u)\in \left(L^2(\Omega)\right)^2\), such that
    \[\left\{
    \begin{aligned}
        (\Delta+k^2n_0)u&=-k^2(n_1-n_0)h\text{ in }\Omega,\\
        (\Delta+k^2n_1)h&=0\text{ in }\Omega,\\
        u&=0\text{ on }\partial\Omega,\\
        \dn u&=0\text{ on }\partial\Omega.\\
    \end{aligned}
    \right.\]
\end{definition}
Since we want to avoid these "{pathological}" values, it is useful to know that this is a rare case.

\begin{lemma}
If the indices \(n_0\) and \(n_1\) are real-valued, then the set of interior transmission eigenvalue for the pair of indices \((n_0,n_1)\) is discrete.
Furthermore, if there are infinitely many, they only accumulate at \(+\infty\).
\end{lemma}
\begin{proof}
    The proof follows exactly the lines of~\cite[Theorems 4.13 and 4.14]{book.kirsch.08}, by adapting the notations.
\qed\end{proof}
With these two geometrical restrictions, the coercivity of the operator $A$ on the range of~$C\etoile$ is then obtained in lemma~\ref{lem:coercivite.2}, using the following result.

\begin{lemma}\label{lem:coercivite.1}\cite[Lemma 1.17]{book.kirsch.08} 
Let $Y$ be a subset of  a reflexive Banach space $X$ and $A$, $A_0$ : $X\etoile \to X$ be linear and bounded operators such that
\begin{enumerate}
\item\label{hyp:coercivite.1} $\ps{\varphi}{A\varphi} \in \C \setminus (-\infty,0]$ for all $\varphi \neq 0$ in the closure of $Y$
\item\label{hyp:coercivite.2} $\ps{\varphi}{A_0\varphi} \in \R$ and there exists $c_0 >0$ with $\ps{\varphi}{A_0\varphi} \geqslant c_0 \norme{\varphi}$ for all $\varphi$ in~$Y$
\item\label{hyp:coercivite.3} $A-A_0$ is compact.
\end{enumerate}
Then, there exists $c>0$ such that for all $\varphi \in Y$ it holds
$\abs{\ps{A\varphi}{\varphi}} \geqslant c \norme{\varphi}^2_{X\etoile}.$
\end{lemma}
Finally, we group the properties of operator $A$ in the following lemma.

\begin{lemma}\label{lem:coercivite.2}
Under assumption~\ref{hyp:coercivite}, if $k$ is not an interior transmission eigenvalue in the sense of definition~\ref{def:transmission},
then the operator $A$, defined by~(\ref{def:A}), is coercive on the range of $C\etoile$, defined in theorem~\ref{thm:omega.1}.
Namely, there exists $c_1>0$ such that  
\begin{equation}\label{eq:encadrement.2}
c_1\norme{\mathbf{1}_{\Omega}C\etoile\Psi}_{L^2(\Omega)} \leqslant \abs{\ps{AC\etoile\Psi}{C\etoile\Psi}_{L^2(\Omega)}}^{\frac{1}{2}},\ \forall \Psi \in L^2(S^{d-1}).
\end{equation}
\end{lemma}
\begin{proof}
This is a straightforward application of lemma~\ref{lem:coercivite.1} with $Y=\range{C\etoile}$ and $X=L^2(\Omega)$.
The required assumptions have been partially shown~\cite[lemma 5.3]{art.grisel.12} but, for convenience, we give a complete proof.
 
\begin{enumerate}   
\item 
 Choose \(\varphi \in \range{C\etoile}\).
    By definition of operator $C$, this is a total field for the refraction index \(n_0\). 
    Hence, there exists an incident field \(f\) such that \(\varphi = \mathcal{T}_{n_0}(f)\).
    Let us set \(u_{n_0} = \varphi\) and \(u_{n_1} = \mathcal{T}_{n_1}(f)\).
    Thus, we obtain \(A\varphi=k^{2}(n_1-n_0)u_{n_1}\).   
    Moreover, choosing \(R>0\) such that the ball \(B_R\) of radius \(R\) contains \(\Omega\), it holds that
    \begin{align*}
        \lefteqn%
        {\int_{\Omega}k^{2}(n_1-n_0)u_{n_1}(\overline{u_{n_0}-u_{n_1}})}\\&
        =\int_{B_R}(\Delta+k^2n_0)(u_{n_1}-u_{n_0})(\overline{u_{n_1}-u_{n_0}})\\
        &=\int_{B_R}k^2n_0|u_{n_1}-u_{n_0}|^2 - |\grad(u_{n_1}-u_{n_0})|^2%\\&\qquad
            +\int_{S_R}(\overline{u_{n_1}-u_{n_0}})\dn (u_{n_1}-u_{n_0}).
        \end{align*}
        By letting \(R\) go to infinity, it comes
        \begin{align*}
        \lefteqn%
         {\int_{\Omega}k^{2}(n_1-n_0)u_{n_1}(\overline{u_{n_0}-u_{n_1}})}\\&
        =\int_{\R^n}k^2n_0|u_{n_1}-u_{n_0}|^2 - |\grad(u_{n_1}-u_{n_0})|^2%\\&\qquad
            +ik\abs{\gamma}^2\int_{S^{d-1}}|u_{n_1}\I-u_{n_0}\I|^2.
    \end{align*}
    Hence, taking the imaginary part yields
    \[Im \int_{\Omega}k^{2}(n_1-n_0)u_{n_1}\overline{u_{n_0}}=k\abs{\gamma}^2\int_{S^{d-1}}|u_{n_1}\I-u_{n_0}\I|^2.\]
    This shows that \(Im \langle A\varphi,\varphi\rangle \geqslant 0\) and if this quantity is vanishing, we deduce that \(u_{n_1}\I=u_{n_0}\I\).
        Moreover, out of $\Omega$ we have \[(\Delta+k^2n_0)u_{n_0} = (\Delta+k^2n_1)u_{n_1} = (\Delta+k^2n_0)u_{n_1}.\]
        As a consequence, the unique continuation principle \cite[theorem 8.6]{book.colton.1} yields \(u_{n_1}=u_{n_0}\) out of \(\Omega\).
    So, the quantity $w \pardef (u_{n_1}-u_{n_0})$ has its support included in \(\Omega\) and satisfies \((\Delta+k^2n_0)w=-(n_1-n_0)u_{n_1}\).
    If \(k\) is not a transmission eigenvalue, we then have \(w=0\) and  $u_{n_1}\rest{\Omega}=0$.
    Finally, this implies that  $u_{n_1}=0$ by the unique continuation principle and all these results are extended to \(\overline{\range{C\etoile}}\) by continuity to prove item~\ref{hyp:coercivite.1}.
\item
  Furthermore, we also have that
    \begin{align*}
        \langle A\varphi,\varphi\rangle
            &=\int_{\Omega}k^{2}(n_1-n_0)|u_{n_0}|^2+\int_{\Omega}k^{2}(n_1-n_0)(u_{n_1}-u_{n_0})\overline{u_{n_0}}\\
            &=\langle A_0\varphi,\varphi \rangle + \langle K\varphi,\varphi \rangle,
    \end{align*}
    with \(A_0=k^{2}(n_1-n_0)I\) and \(K=k^{2}(n_1-n_0)(\mathcal{T}_{n_1}\mathcal{T}_{n_0}\inv-I)\).
    Under assumption~\ref{hyp:coercivite}, \(A_0\) is clearly coercive and self-adjoint.
    \item
    Moreover, \(K=k^{2}(n_1-n_0)(\mathcal{T}_{n_1}-\mathcal{T}_{n_0})\mathcal{T}_{n_0}\inv\).
    Yet, it is known from the Lippmann-Schwinger equation~\cite[equation (8.12)]{book.colton.1} that $\mathcal{T}_{n} = id - T\mathcal{T}_{n}$, where $T$ is some compact operator.
    Thus, $(\mathcal{T}_{n_1}-\mathcal{T}_{n_0})$ is compact, and so is $K$.
\end{enumerate}
\qed\end{proof}

\subsection{Characterization of the defects from the measurements}

We are now able to state an explicit localization of the defects in the form of a constrained optimization problem that extends the characterization proposed in~\cite{art.kirsch.00}.

\begin{theorem}\label{thm:minimisation}
Assume that $k$ is not an interior transmission eigenvalue for the indices \(n_0\) et \(n_1\), following  definition~\ref{def:transmission}, 
that these indices are contrasted following assumption~\ref{hyp:coercivite}
and that we have full bi-static data (i.e. $u_{n_0}\I$ and $u_{n_1}\I$ are known over \(S^{d-1} \times S^{d-1}\)).

We can then define the value
    \begin{equation}\label{eq:I.1}
        \mathcal M_{W}(z) \pardef \inf\left\{ f_{W}(\Psi),\ \Psi \in L^2(S^{d-1}) \text{ and } {\ps{\Psi}{\overline{u_{n_0}(\cdot,z)}}}_{L^2(S^{d-1})} = 1\right\},
    \end{equation}
    and for each point \(z \in \R^d\) we have
    \begin{equation}\label{eq:minimisation}
        z \in \Omega \iff \mathcal M_{W}(z) >0.
    \end{equation}
\end{theorem}
\begin{proof}
Theorem~\ref{thm:omega.1} characterizes the domain of the defects $\Omega$ by
\[z \in \Omega \iff \overline{u_{n_0}(\cdot,z)} \in \range{C\mathbf{1}_{\Omega}}.\]
The result~(\ref{eq:minimisation}) is then a direct consequence of corollary~\ref{cor:image.2} with $f = f_W$, $\varphi = \overline{u_{n_0}(\cdot,z)}$, $L = C\mathbf{1}_\Omega$ and $H_2 = L^2(S^{d-1})$.
So what is left to prove is that the double inequality~(\ref{eq:encadrement.1}) holds.
Now, we deduce from~(\ref{eq:factorisation.2}) that the right inequality comes from the boundedness of the operator $A$. The left one was established in lemma~\ref{lem:coercivite.2}.
\qed\end{proof}

\section{Numerical methods for the computation of the infimum's value}\label{sec:iteratives}

In theorem~\ref{thm:minimisation}, we have expressed the localization of the defects as a pointwise binary test, taking the form of a constrained optimization problem.
We are now interested in the numerical computation of the values of the function $\mathcal M_{W}(z)$ defined in~(\ref{eq:I.1}).
We thus explicit, and then test, two usual minimization algorithms working on this problem: the steepest descent and the gradient projection.

\subsection{Algorithms}

Both minimization methods we are going to use require explicit gradients of the cost function.
To simplify their expression, we will consider the minimization of the following form on $L^2(S^{d-1})$, defined for any $\Psi$ by
\[f_W^4(\Psi) \pardef \big(f_W(\Psi) \big)^4 = \abs{\ps{W\Psi}{\Psi}_{L^2(S^{d-1})}}^2.\]
Since we only want to know if the infimum is vanishing, this gives results equivalent to~(\ref{eq:I.1}) and we thus have to evaluate
\begin{equation}\label{eq:I.2}
         \big(\mathcal M_W(z)\big)^4 \pardef \inf\left\{ f_W^4(\Psi),\ \Psi \in L^2(S^{d-1}) \text{ and }{\ps{\Psi}{\overline{u_{n_0}(\cdot,z)}}}_{L^2(S^{d-1})} = 1\right\}.
    \end{equation}
%Of course, it holds that $\mathcal M_{W}(z) = 0 \iff  \mathcal M_W^4(z)=0$.
\begin{remark}
The function $\mathcal M_W(z)$ has been proved to  vanish outside of the defects and yet, it can be seen that the value 0 can never be attained while satisfying the constraint.
	
To show this, let $\Psi_\star \in L^2(S^{d-1})$  be such that \(f_W^4(\Psi_\star) = 0\).
Then, from  factorization~(\ref{eq:factorisation.1}) and inequalities~(\ref{eq:encadrement.2}), it holds that $ f_W^4(\Psi_\star) = \ps{AC\etoile\Psi_\star}{C\etoile\Psi_\star}^2 \geqslant \norme{\mathbf{1}_\Omega C\etoile\Psi_\star}^4$, that is $C\etoile{\Psi_\star}\rest{\Omega} =0$.
Furthermore, it is easy to see that $C\etoile{\Psi_\star}$ satisfies the Helmholtz equation~(\ref{eq:Helmholtz}) on $\R^d$ with $n = n_0$.
The unique continuation principle then yields $C\etoile{\Psi_\star} =0$, and thus ${\Psi_\star} = 0$ by injectivity of $C\etoile$, which has been shown in the proof of~\cite[Proposition 5.4]{art.grisel.12}.
As a consequence, $\Psi_\star$ can not satisfy ${\ps{\Psi_\star}{\overline{u_{n_0}(\cdot,z)}}} = 1$.

This points out  that  numerical approximations of a vanishing \textit{inf} might be mistaken with exact non-zero \textit{inf} values.
Some care will thus have to be taken to set them apart, so that the plot of $\mathcal M_W(z)$ can be used to localize the defects.
\end{remark}

We now turn to the feasible set.
For \(z \in \R^d\), let us denote by $\mathcal C_z$ the affine hyperplane
\begin{equation*}%\label{eq:contraintes}
        \mathcal{C}_{z} \pardef \left\{\Psi+\frac{\overline{u_{n_0}(\cdot,z)}}{\norme{u_{n_0}(\cdot,z)}^2},\  \Psi \in L^2(S^{d-1}),\  \ps{\Psi}{\overline{u_{n_0}(\cdot,z)}}_{L^2(S^{d-1})} = 0 \right\}.
\end{equation*}
The orthogonal projection on  $\mathcal C_z$, denoted by $P_{\mathcal C_z}$, is defined on \(L^2(S^{d-1})\) by the affine mapping
    \begin{equation}\label{eq:proj}
        P_{\mathcal{C}_z} \Psi \pardef \Psi-\frac{\ps{\Psi}{\overline{u_{n_0}(\cdot,z)}}}{\norme{u_{n_0}(\cdot,z)}^2}\overline{u_{n_0}(\cdot,z)} +  \frac{\overline{u_{n_0}(\cdot,z)}}{\norme{u_{n_0}(\cdot,z)}^2}.
    \end{equation}

Hence, finding an infimum of \(f_W^4\) over \(\mathcal{C}_z\) is equivalent to looking for the infimum over \(L^2(S^{d-1})\) of  
\begin{equation*}%\label{def:gz}
    P_{f_{W}^4} \pardef f_{W}^4 \rond P_{\mathcal{C}_z}.
\end{equation*}
We can therefore compute a minimizing sequence $x_n$ for the form $P_{f_{W}^4} $ by any unconstrained optimization method.
For convenience, all gradients and hessians are calculated in section~\ref{sec:gradients} and their finite dimension formulation is given in section~\ref{sec:dimension.finie}.
It then follows from theorem \ref{thm:minimisation} that if \(\big(P_{f_{W}^4} (x_n)\big)_n\) goes to 0, the point \(z\) is outside~\(\Omega\), and inside otherwise.
A basic example of descent is presented in algorithm~\ref{algo:descent}:

\begin{algorithm}[H]\label{algo:descent}
\caption{Steepest descent}
\KwIn{\(x_0\) chosen in \(L^2(S^{d-1})\)}
\Repeat{\(\norme{x_{n+1} - x_n}/(1+\norme{x_n})<\varepsilon\)}{
Compute \(\alpha_n\) such that \(P_{f_{W}^4} (x_n-\alpha_n \grad P_{f_{W}^4} (x_n)) < P_{f_{W}^4} (x_n)\);\\
Update \(x_{n+1} \leftarrow x_n-\alpha_n \grad P_{f_{W}^4} (x_n)\);
}
\KwOut{\(P_{f_{W}^4} (x_{N})\)}
\end{algorithm}

Moreover, since the projection on $\mathcal C_z$ is easy to write, we can also consider a gradient projection method~\cite{art.rosen.60}.
As previously, if \(\big(f_{W}^4(x_n)\big)_n\) goes to 0, the point \(z\) is outside \(\Omega\).
The principle of the gradient projection is presented in algorithm~\ref{algo:gp}:

\begin{algorithm}[H]\label{algo:gp}
\caption{Gradient projection}
\KwIn{\(x_0\) chosen in \(\mathcal{C}_z\)}
\Repeat{\(\norme{x_{n+1} - x_n}/(1+\norme{x_n})<\varepsilon\)}{
Compute \(\alpha_n\) such that \(f_{W}^4(x_n-\alpha_n \grad f_{W}^4(x_n)) < f_{W}^4(x_n)\);\\
Project and update \(x_{n+1} \leftarrow P_{\mathcal{C}_z}(x_n-\alpha_n \grad f_{W}^4(x_n))\);
}
\KwOut{\(f_{W}^4(x_{N})\)}
\end{algorithm}

\begin{remark}\label{rem:descente.projection}
Since the projector \(P_{\mathcal{C}_z}\) is an affine map, the proposed steepest descent  and gradient projection methods are very close.
Indeed, we note that by choosing a constant descent step \(\alpha_n=\alpha\) and a starting point \(x_0 \in \mathcal{C}_z\), both algorithms define the same sequence.
The computation of $\alpha_n$ is however done before the projection in algorithm~\ref{algo:gp}, and after it in algorithm~\ref{algo:descent}.
Thus, the values coming up for $\alpha_n$ in each of the proposed algorithms will seemingly be different and produce different sequences $x_n$.
%A popular method to compute $\alpha_n$ is Wolfe's rule~\cite[Chapter 3.4]{book.bonnans.06}.
\end{remark}

\subsection{Numerical validation on a simple case}

\begin{figure}[htbp]
\centering
\includegraphics[height=3cm]{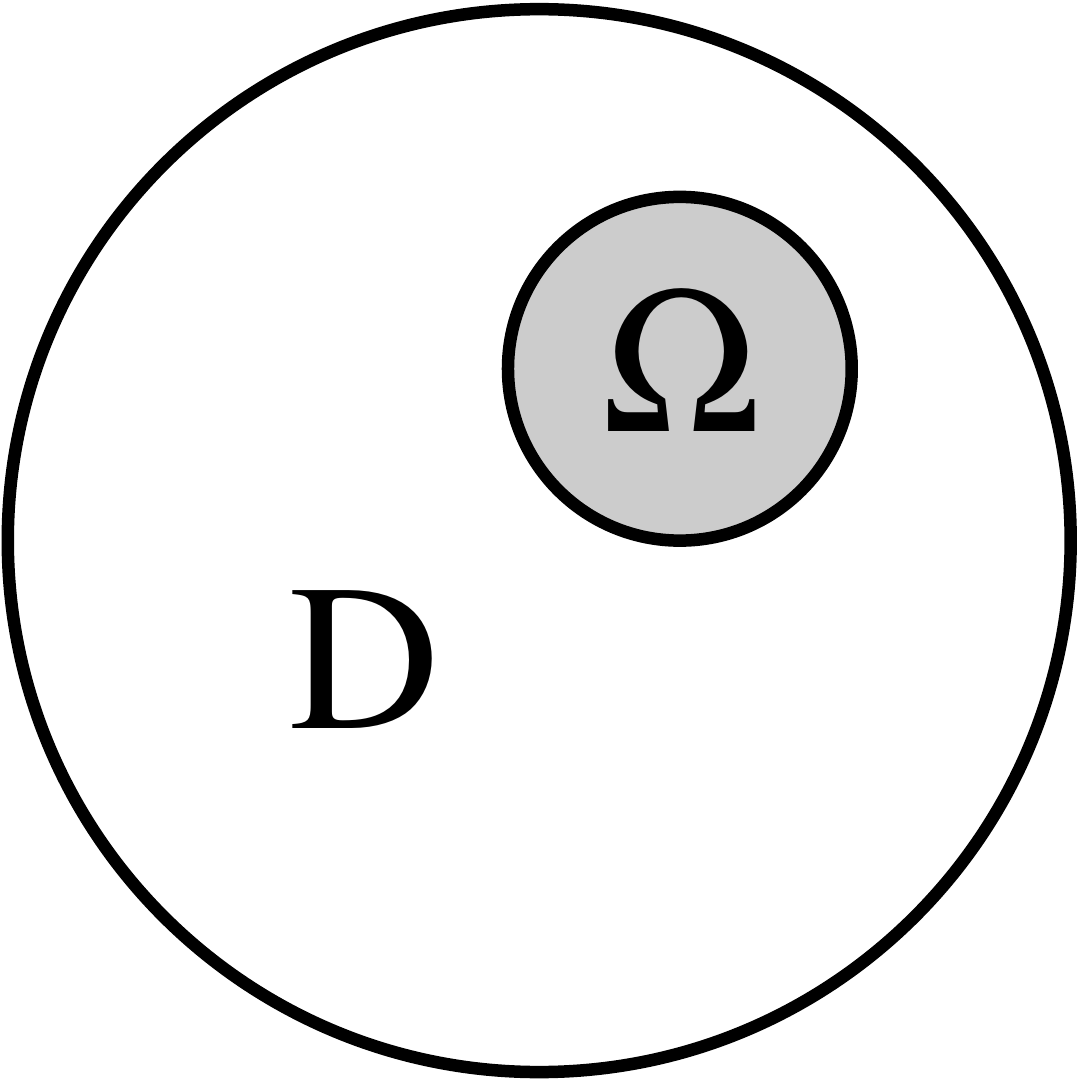}
\caption{A simple study case}
\label{fig:simple.schema}
\end{figure}

We present here some 2D numerical results  in the simple case illustrated on figure~\ref{fig:simple.schema}.
The considered object of support is a disc $D$ of section 2.1 containing defects which are also in the shape of a disc $\Omega$ of section 0.6.
With a fixed wave number \(k=10\), the size of the object is then thrice the wavelength and the size of the defects is approximatively one wavelength.
Also, the reference index $n_0$ takes its values in $[1.56,\,1.84]$ inside $D$ and the perturbed version $n_1$ takes its values in $[2.01,\,2.16]$ inside $\Omega$.
Finally, we used $99$ incoming/measurement directions evenly distributed over $[0,\,2\pi]$.

\begin{figure}[htbp]
\centering

\subfloat[Steepest descent algorithm]{\includegraphics[width=0.45\linewidth]{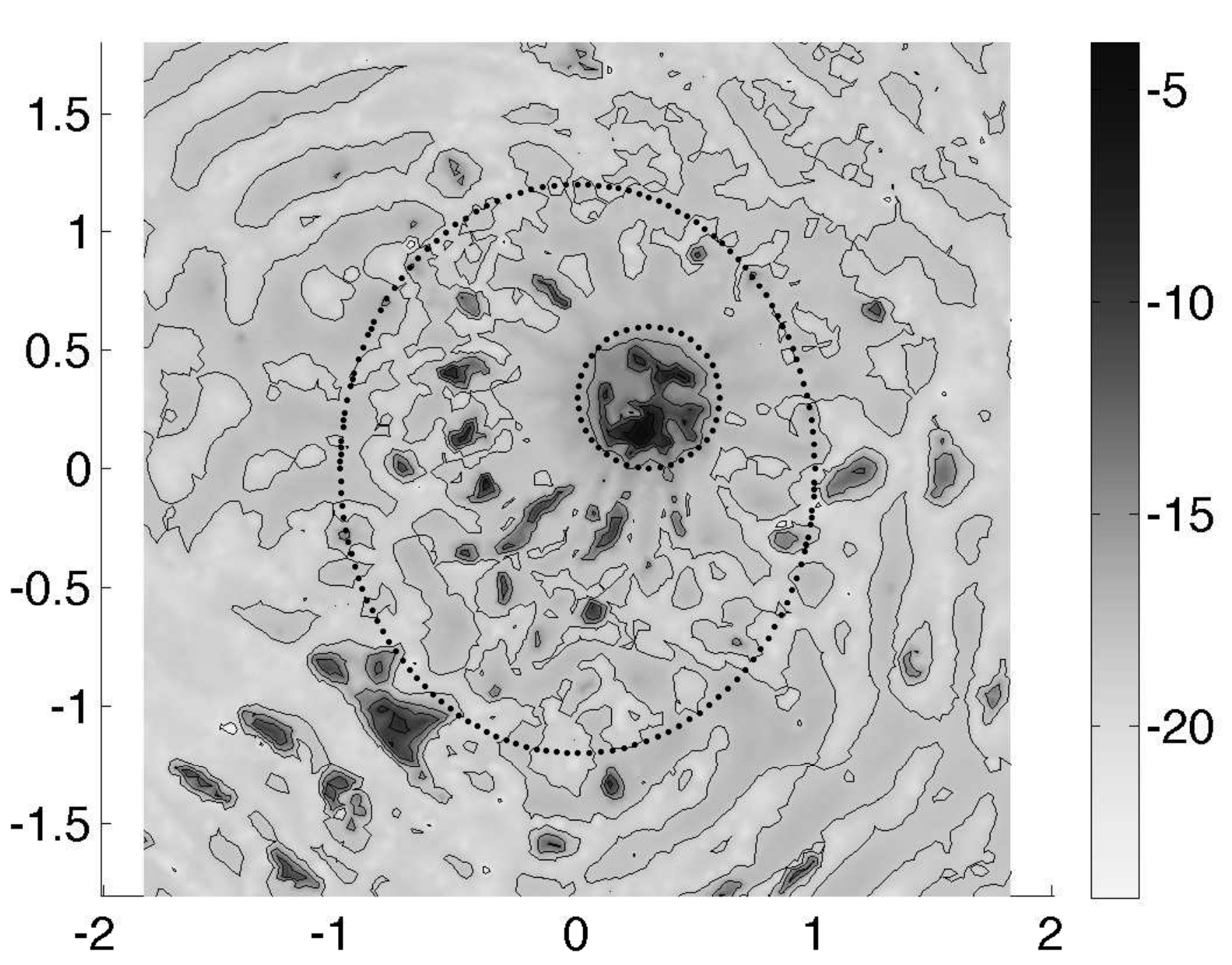}\label{fig:simple.descente}}\hfill
\subfloat[Gradient projection algorithm]{\includegraphics[width=0.49\linewidth]{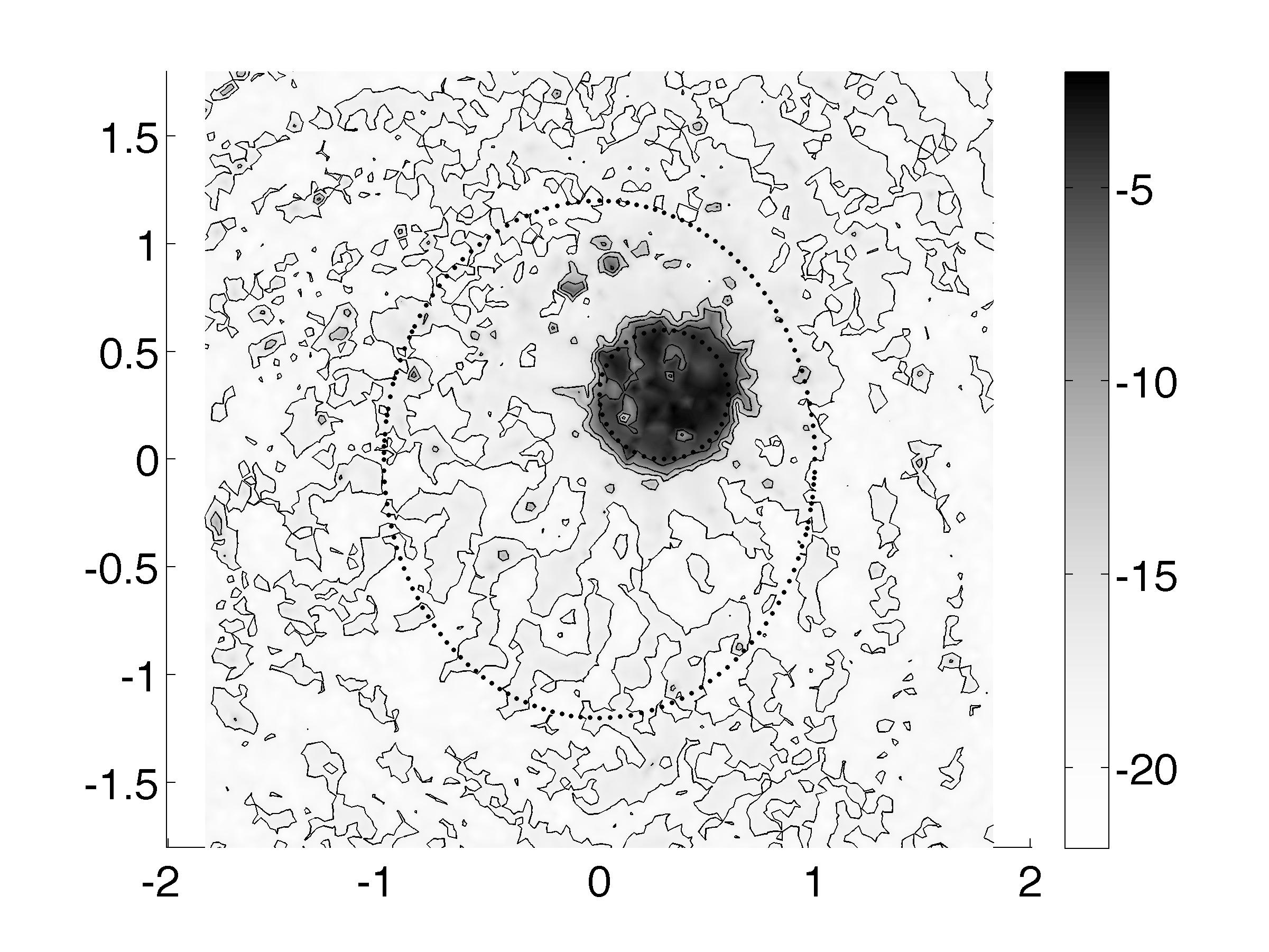}\label{fig:simple.projection}}

\subfloat[Matlab's \texttt{fminunc} function]{\includegraphics[width=0.49\linewidth]{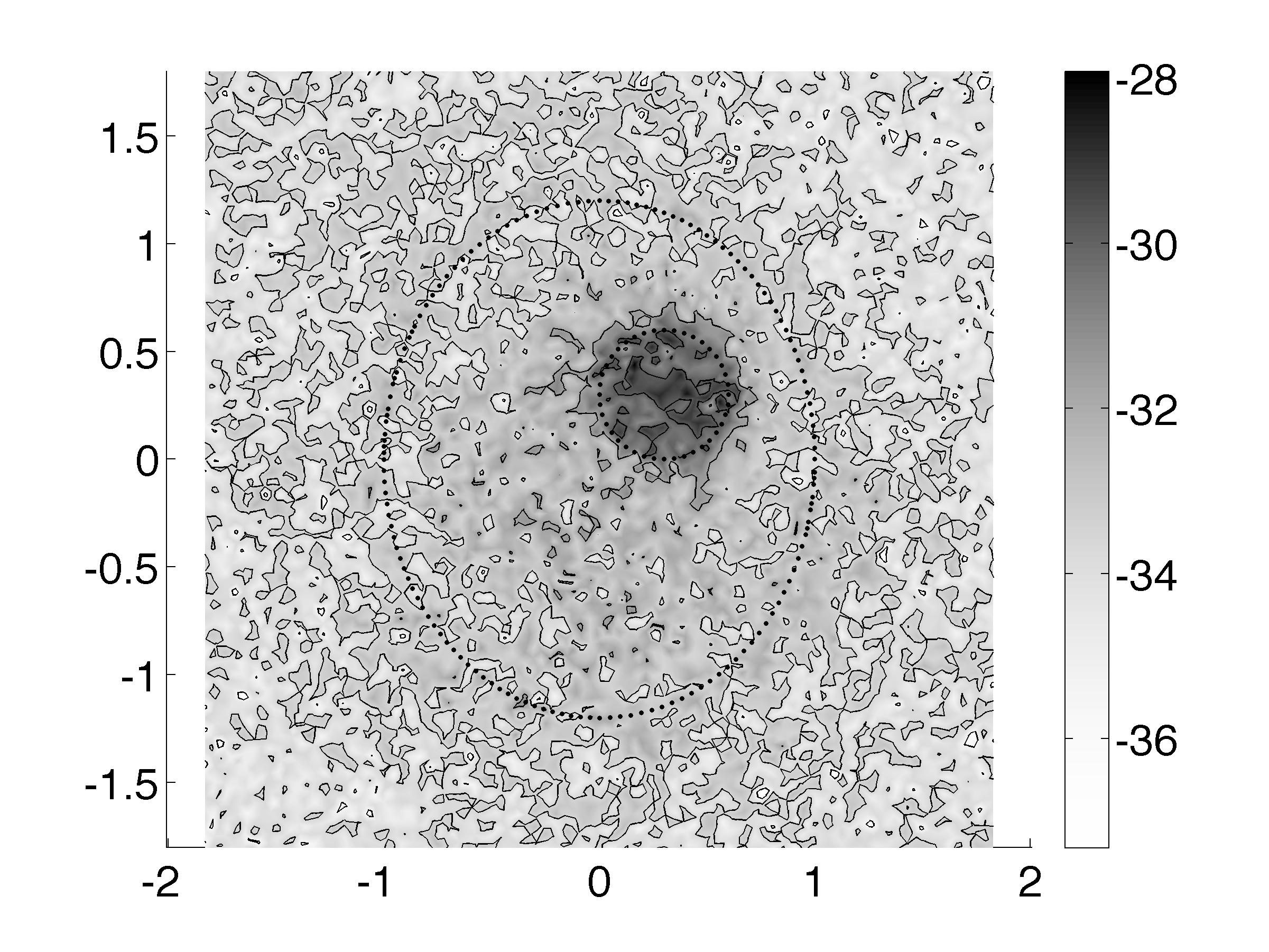}\label{fig:simple.fminunc}}\hfill
\subfloat[Matlab's \texttt{fmincon} function]{\includegraphics[width=0.49\linewidth]{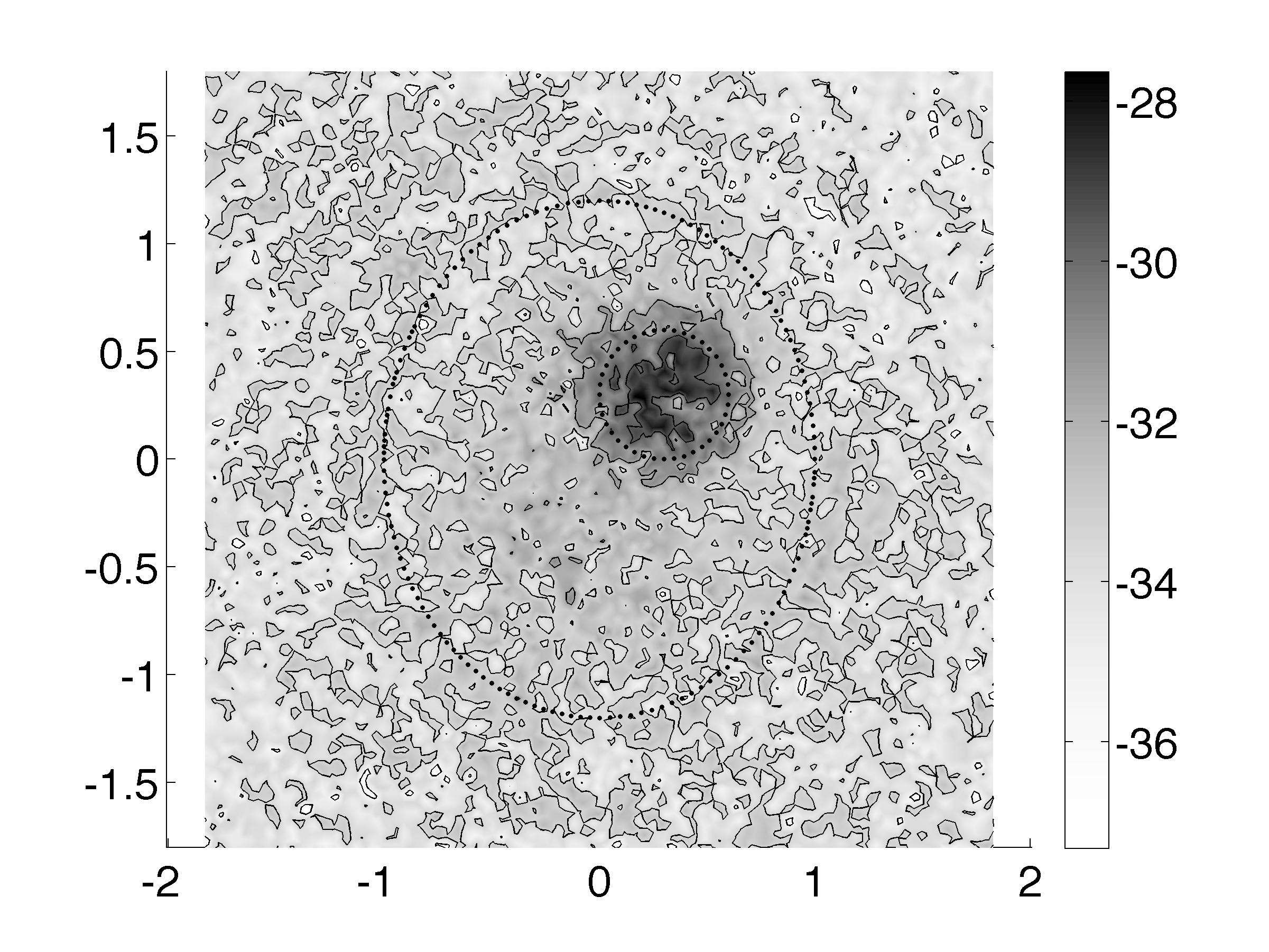}\label{fig:simple.fmincon}}

\medskip

\begin{tabular}{|c|ccc|}
\hline
\textbf{Number of iterations} 
& min & med & max ($\leqslant 400$)\\
\hline
Steepest descent & 11 & 400 & 400\\
Gradient projection & 14 & 120 & 400\\
Matlab's \texttt{fminunc} function & 8 & 17 & 68\\
Matlab's \texttt{fmincon} function & 9 & 19 & 58\\
\hline
\end{tabular}

\caption{Values of $\log_{10}\mathcal M_W(z)$ computed by various optimization methods with a relative tolerance on $x_n$ set to $10^{-9}$ as the only stopping rule}
\label{fig:methodes.iteratives}
\end{figure}

Figure~\ref{fig:methodes.iteratives} displays the infimums's values $\big(\mathcal M_W(z)\big)^4$ in a $\log_{10}$ scale, respectively obtained for each sampling point \(z_i\) (about 8500) with algorithm~\ref{algo:descent} (figure~\ref{fig:simple.descente}) and algorithm~\ref{algo:gp} (figure~\ref{fig:simple.projection}).
Note that these sampling points are unrelated to the finite elements nodes that were used to generate the data.
The linear search for the step length $\alpha$ was done by simple dichotomy on the gradient.
We also present the results obtained with Matlab's \texttt{fminunc} function applied to the form $P_{f_W^4}$ (figure~\ref{fig:simple.fminunc}) and Matlab's \texttt{fmincon} function applied to the form $f_W^4$ over the feasible set $\mathcal C_{z_i}$ (figure~\ref{fig:simple.fmincon}).
The \texttt{fminunc} and  \texttt{fmincon} functions are based on the interior-reflective Newton method described in~\cite{art.coleman.96}.
Furthermore, for each sampling point and each algorithm, the sequence has been initialized by \(x_0 = P_{\mathcal{C}_{z_i}}(0) = {\overline{u_{n_0}(\cdot,z_i)}}/{\norme{u_{n_0}(\cdot,z_i)}^2} \in \mathcal{C}_{z_i}\).
Indeed, it seems natural to start with a point already satisfying the constraint.

As can be expected, Matlab's functions give much faster results, but we note that even the very basic algorithms we proposed yield acceptable results.
This shows that the iterative optimization approach resulting in theorem~\ref{thm:minimisation} can produce a satisfactory localization of the defects, but for a computational cost that is not controlled at this point.

\begin{figure}[htbp]
\centering

\subfloat[Steepest descent algorithm]{\includegraphics[width=0.49\linewidth]{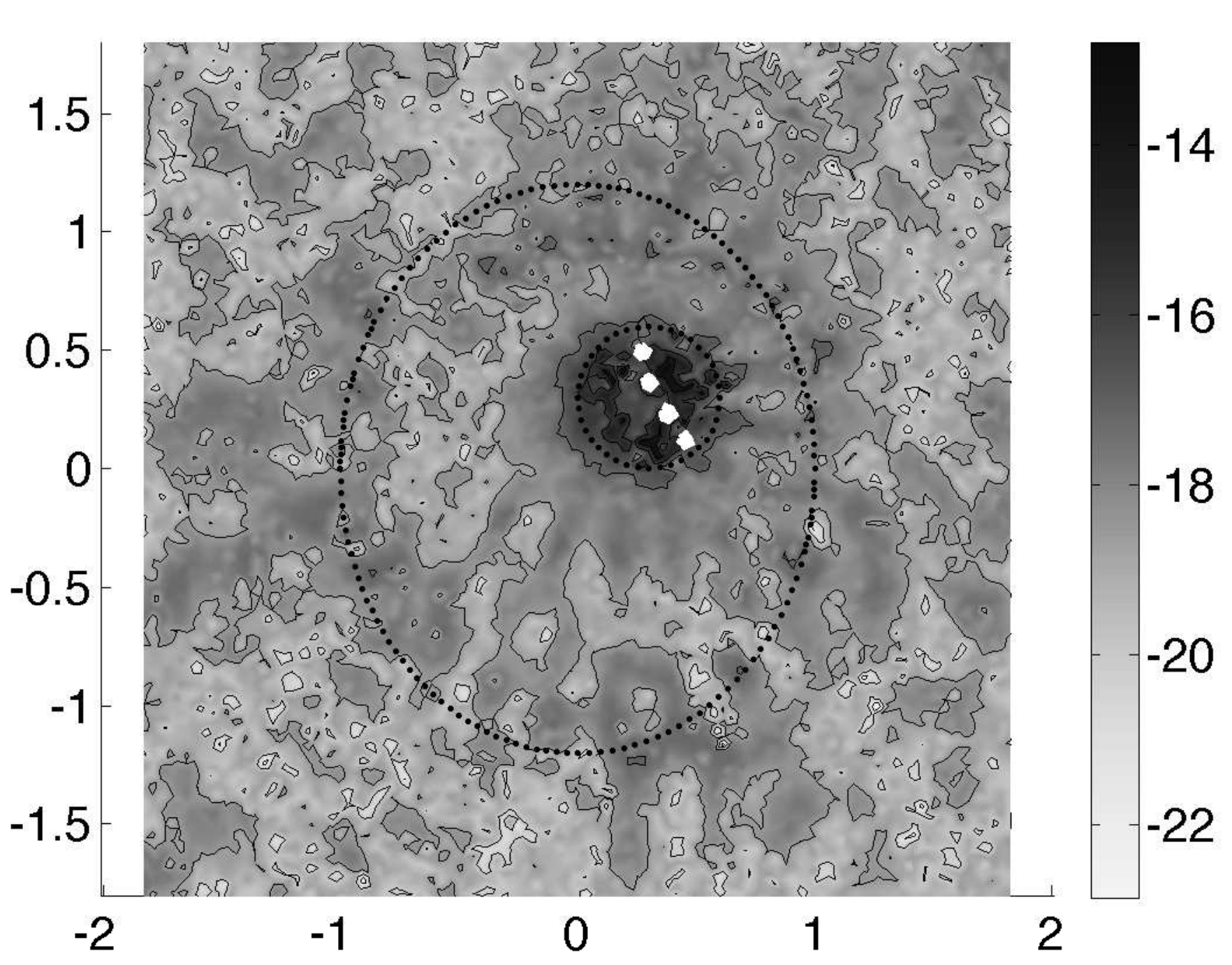}\label{fig:simple.descente.bruit.40}}\hfill
\subfloat[Gradient projection algorithm]{\includegraphics[width=0.49\linewidth]{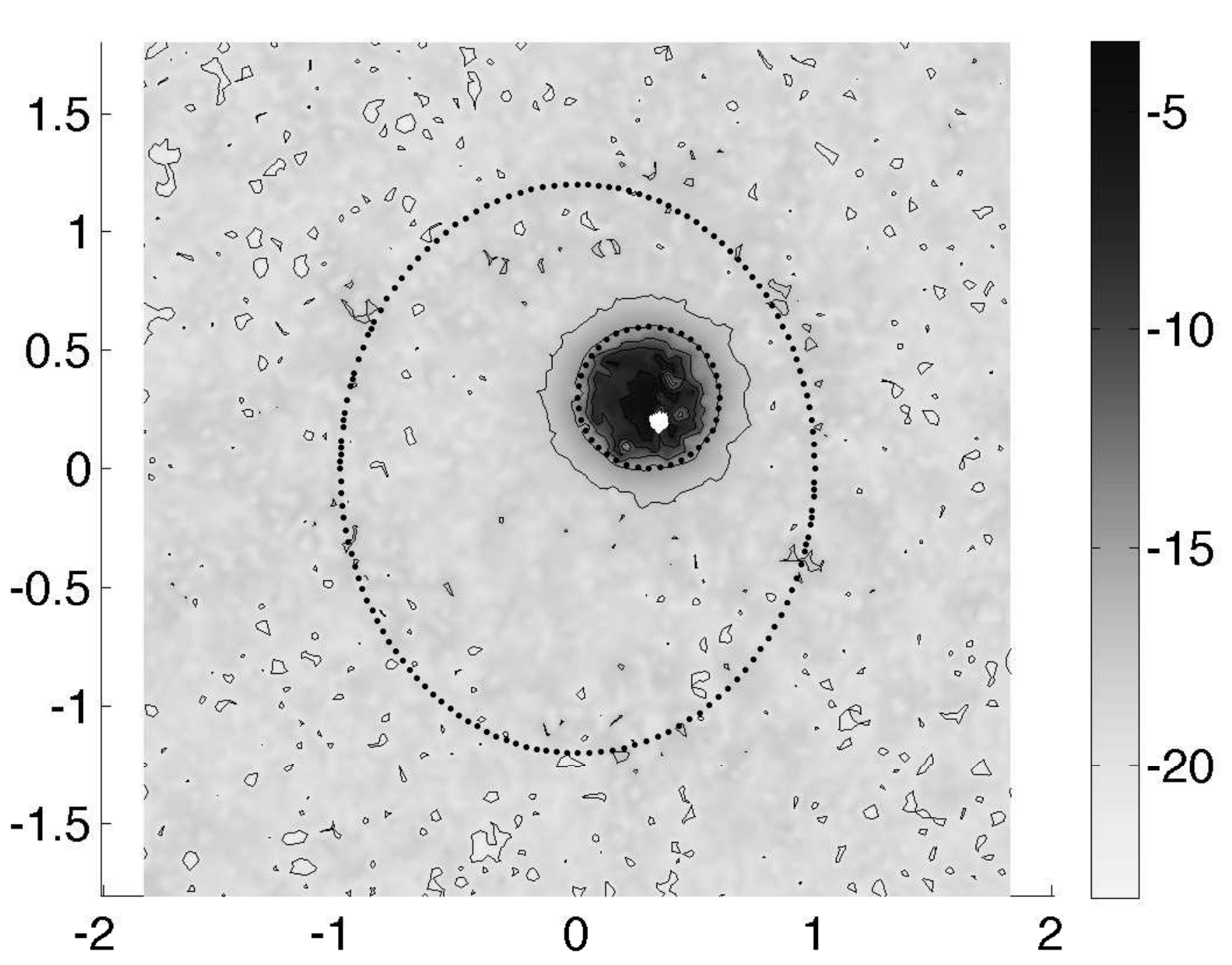}\label{fig:simple.projection.bruit.40}}

\subfloat[Matlab's \texttt{fminunc} function]{\includegraphics[width=0.49\linewidth]{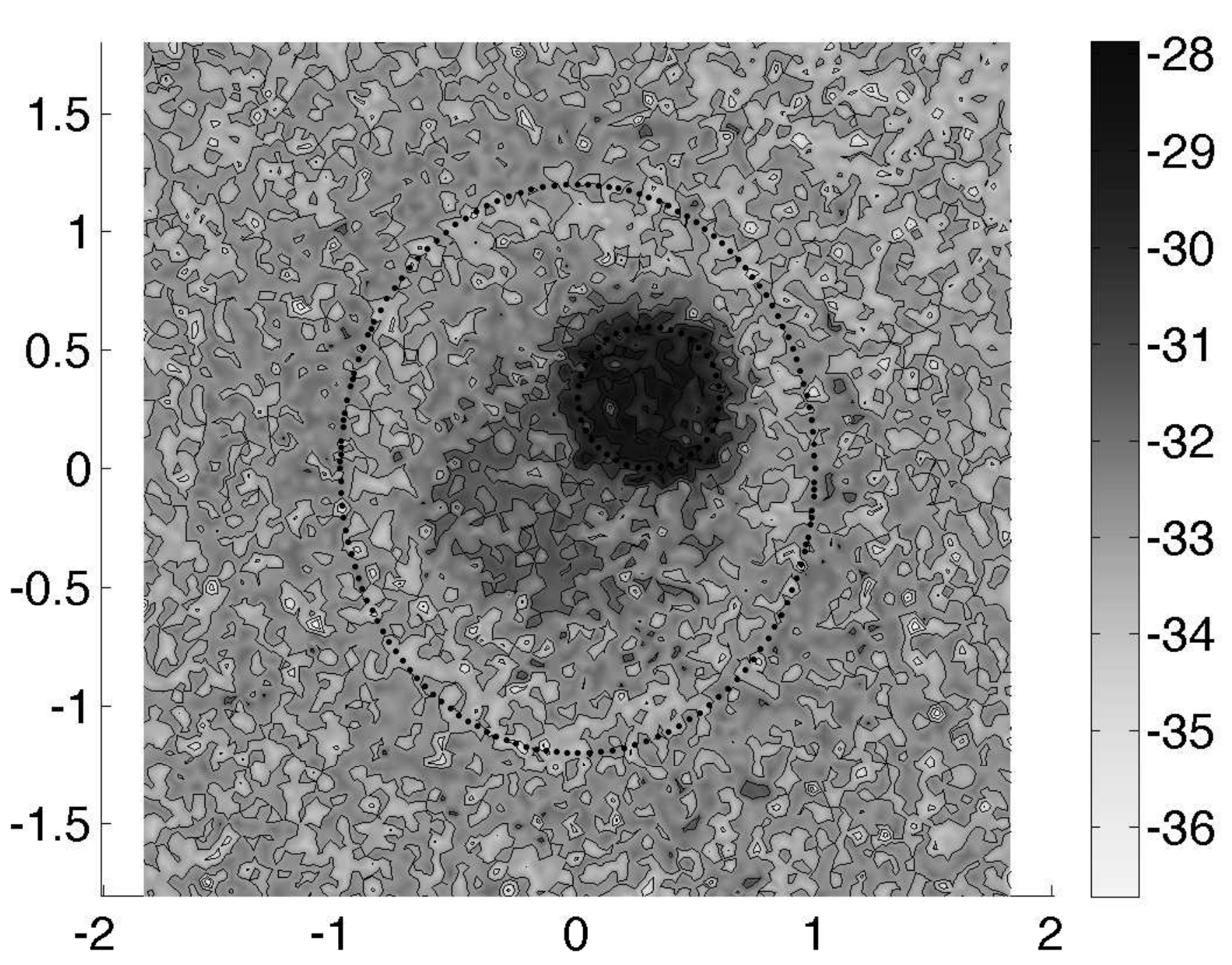}\label{fig:simple.fminunc.bruit.40}}\hfill
\subfloat[Matlab's \texttt{fmincon} function]{\includegraphics[width=0.49\linewidth]{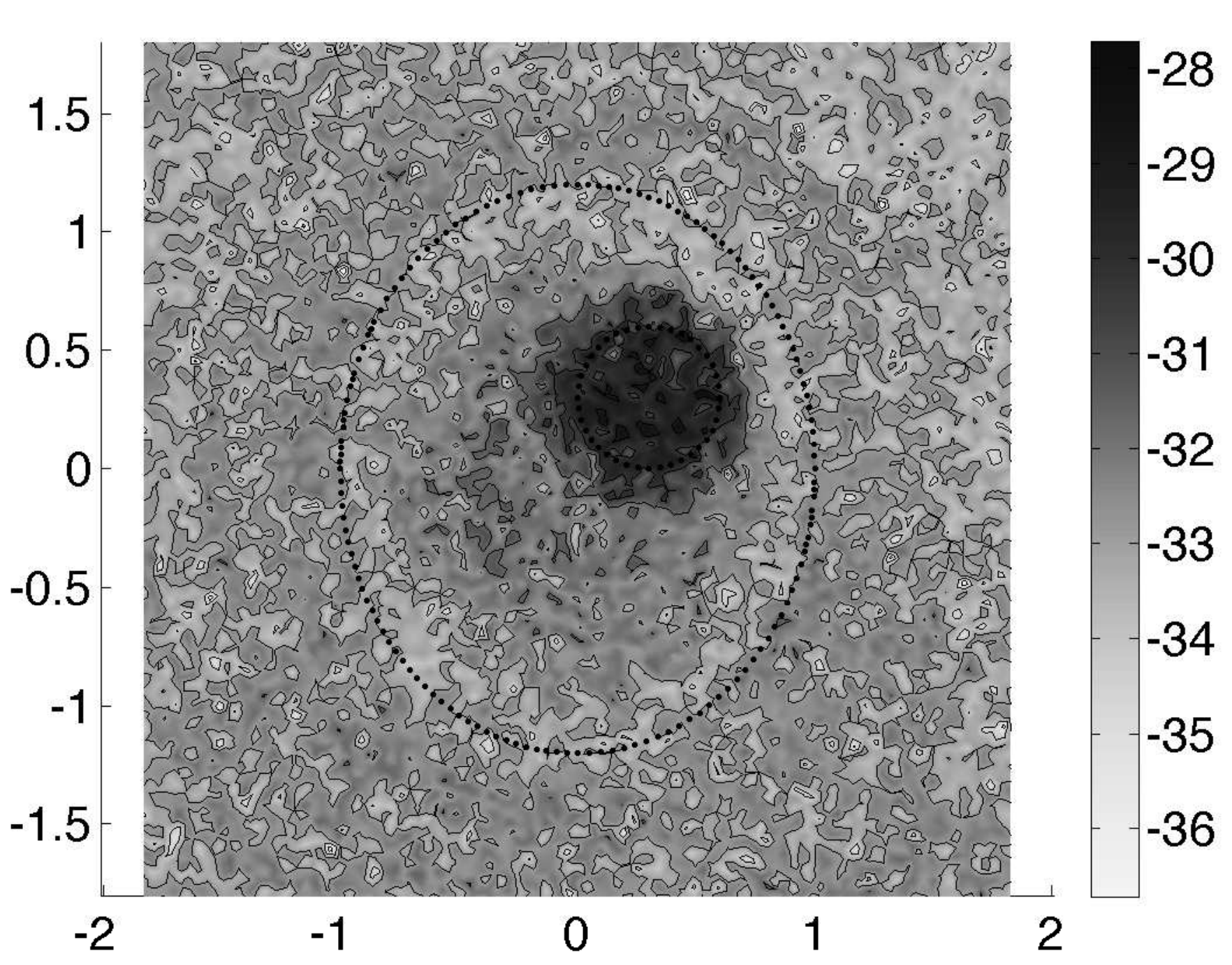}\label{fig:simple.fmincon.bruit.40}}

\medskip

\begin{tabular}{|c|ccc|}
\hline
\textbf{Number of iterations} 
& min & med & max ($\leqslant 400$)\\
\hline
Steepest descent & 5 & 10 & 400\\
Gradient projection & 6 & 12 & 145\\
Matlab's \texttt{fminunc} function & 6 & 10 & 14\\
Matlab's \texttt{fmincon} function & 8 & 12 & 18\\
\hline
\end{tabular}

\caption{Values of $\log_{10}\mathcal M_W(z)$ with $10\%$ uniform noise added to the measurements}
\label{fig:methodes.iteratives.bruit.10}
\end{figure}

To go further, we also tested the sensitivity with respect to the data, to take into account simulation  or measurement inaccuracy.
This was done by adding uniform random noise to the measurements and thus, using $u^\varepsilon_{n_1}$ such that $\norme{u^\varepsilon_{n_1}-u\I_{n_1}}<\varepsilon\norme{u\I_{n_1}}$.
The optimization approach turns out to be quite robust regarding this criterion.
Indeed, figure~\ref{fig:methodes.iteratives.bruit.10} shows that acceptable results are still obtained with $10\%$ relative noise, with a seemingly better visualization for the gradient projection algorithm.
As this is sometimes the case, we also note that adding some noise has a slightly regularizing effect   that visibly enhances convergence speed  for the basic algorithms~\ref{algo:descent} and~\ref{algo:gp}.

\begin{figure}[htbp]
\centering

\subfloat[Steepest descent algorithm]{\includegraphics[width=0.49\linewidth]{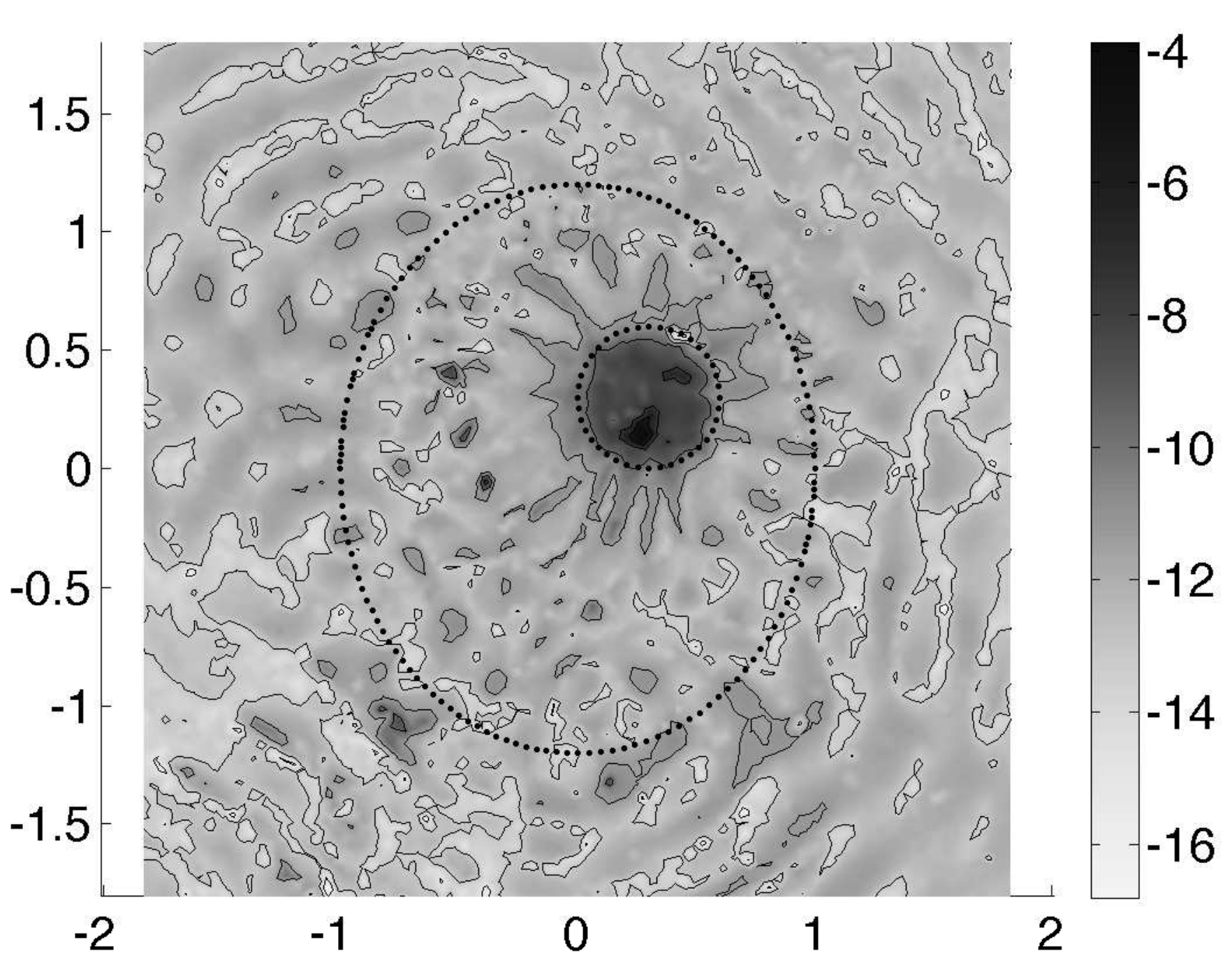}\label{fig:simple.descente.6}}\hfill
\subfloat[Gradient projection algorithm]{\includegraphics[width=0.49\linewidth]{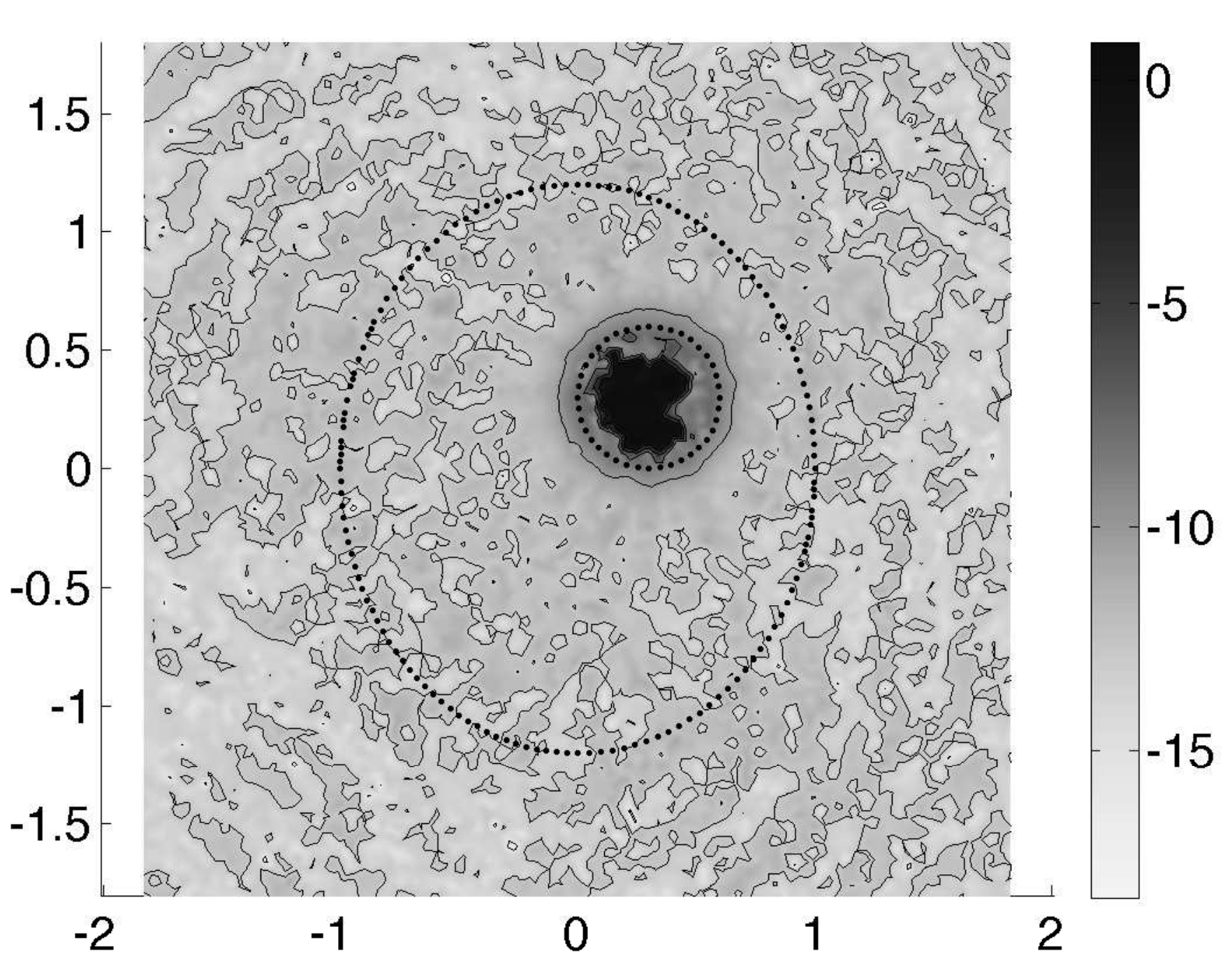}\label{fig:simple.projection.6}}

\subfloat[Matlab's \texttt{fminunc} function]{\includegraphics[width=0.49\linewidth]{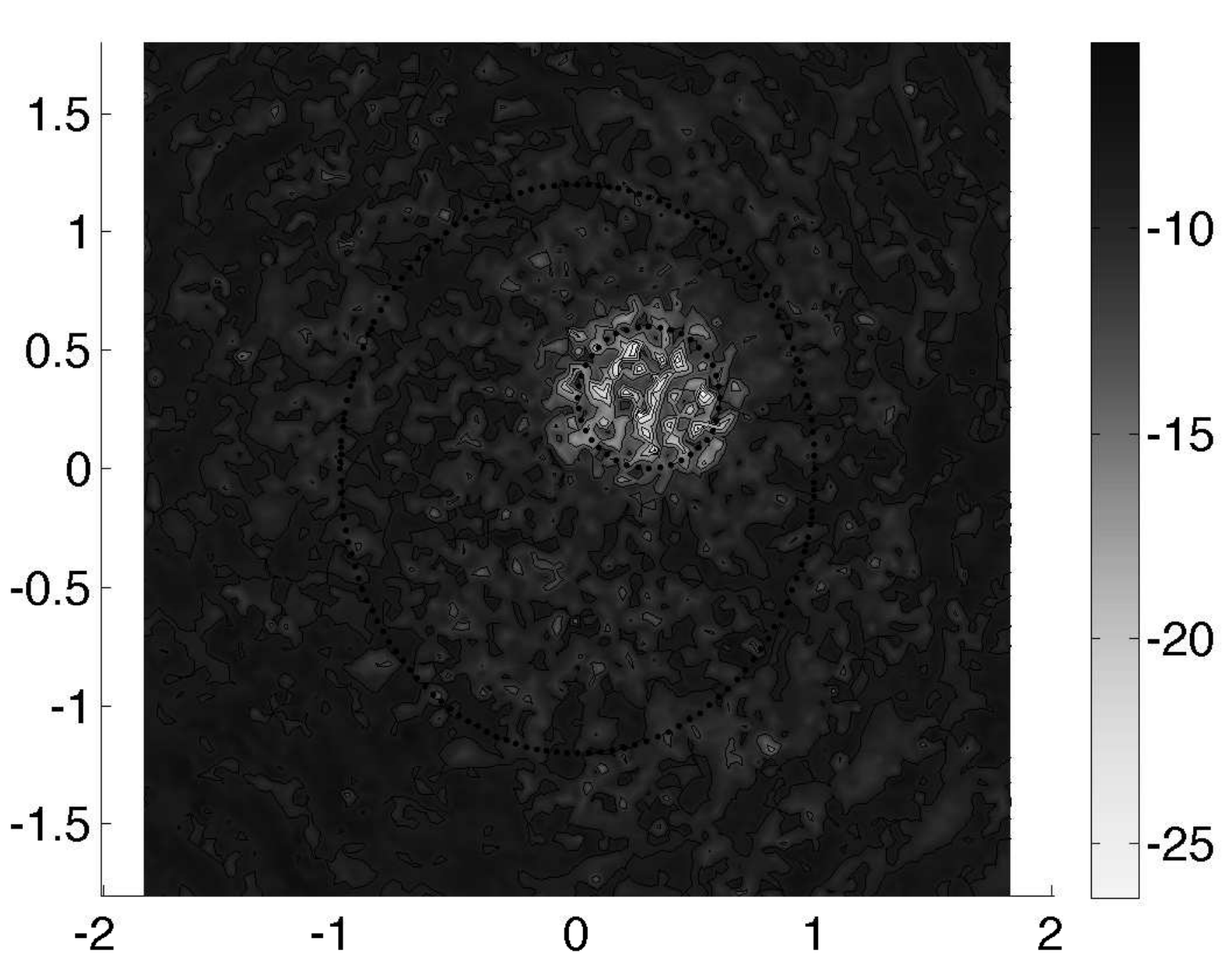}\label{fig:simple.fminunc.6}}\hfill
\subfloat[Matlab's \texttt{fmincon} function]{\includegraphics[width=0.49\linewidth]{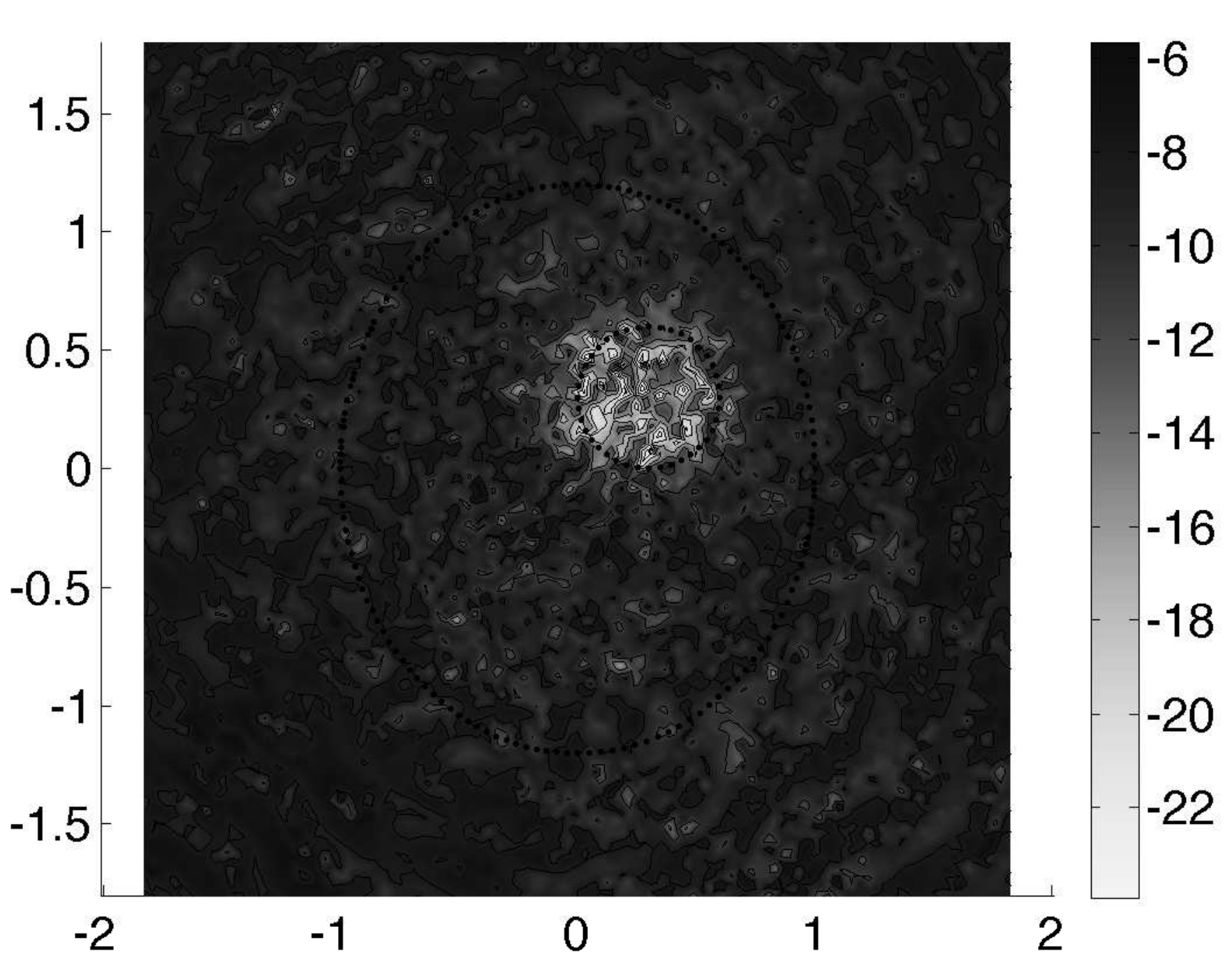}\label{fig:simple.fmincon.6}}

\medskip

\begin{tabular}{|c|ccc|}
\hline
\textbf{Number of iterations} 
& min & med & max ($\leqslant 400$)\\
\hline
Steepest descent & 4 & 46 & 400\\
Gradient projection & 5 & 29 & 400\\
Matlab's \texttt{fminunc} function & 4 & 9 & 23\\
Matlab's \texttt{fmincon} function & 7 & 9 & 26\\
\hline
\end{tabular}

\caption{Localization of the defects with the relative tolerances set to $10^{-6}$}
\label{fig:methodes.iteratives.6}
\end{figure}

\begin{remark}\label{rmk:stopping}
The results in figures~\ref{fig:methodes.iteratives} and~\ref{fig:methodes.iteratives.bruit.10}  were obtained by considering only the relative variation on $x_n$, as proposed in algorithm~\ref{algo:descent}:
$\norme{x_{n+1} - x_n} / (1+\norme{x_n}) < 10^{-9}.$
Yet, many optimization algorithms rely on multiple stopping criterions, including a relative variation tolerance with respect to the cost function's values.
We see here that these values are very small and thus hardly usable in a stopping rule.
The same goes for the gradient and first order optimality criteria.
As a consequence, we had to set the relative tolerance regarding the cost function to $10^{-32}$ in Matlab's optimization functions, so that this condition would never be triggered.
With the tolerances for the relative variation regarding $x_n$ and the cost function set to their default ($10^{-6}$, see Matlab's help), we see in figure~\ref{fig:methodes.iteratives.6} that Matlab's functions produce the opposite result to what was expected.
Indeed, we see that the values of $\mathcal M_W(z)$ outside $\Omega$ are close to zero but higher than the ones inside $\Omega$.
On the other hand, the simple algorithms presented in the previous section still yield the expected localizations, and at a lower computational time. 
This highlights that a special care has to be  taken regarding the stopping rule, as we compare the minimal values issued from different minimization problems.
\end{remark}

\subsection{Experimentations on a non-trivial absorbing example}

\begin{figure}[htbp]
\centering

\subfloat[Reference index's real part]{\includegraphics[width=0.49\linewidth]{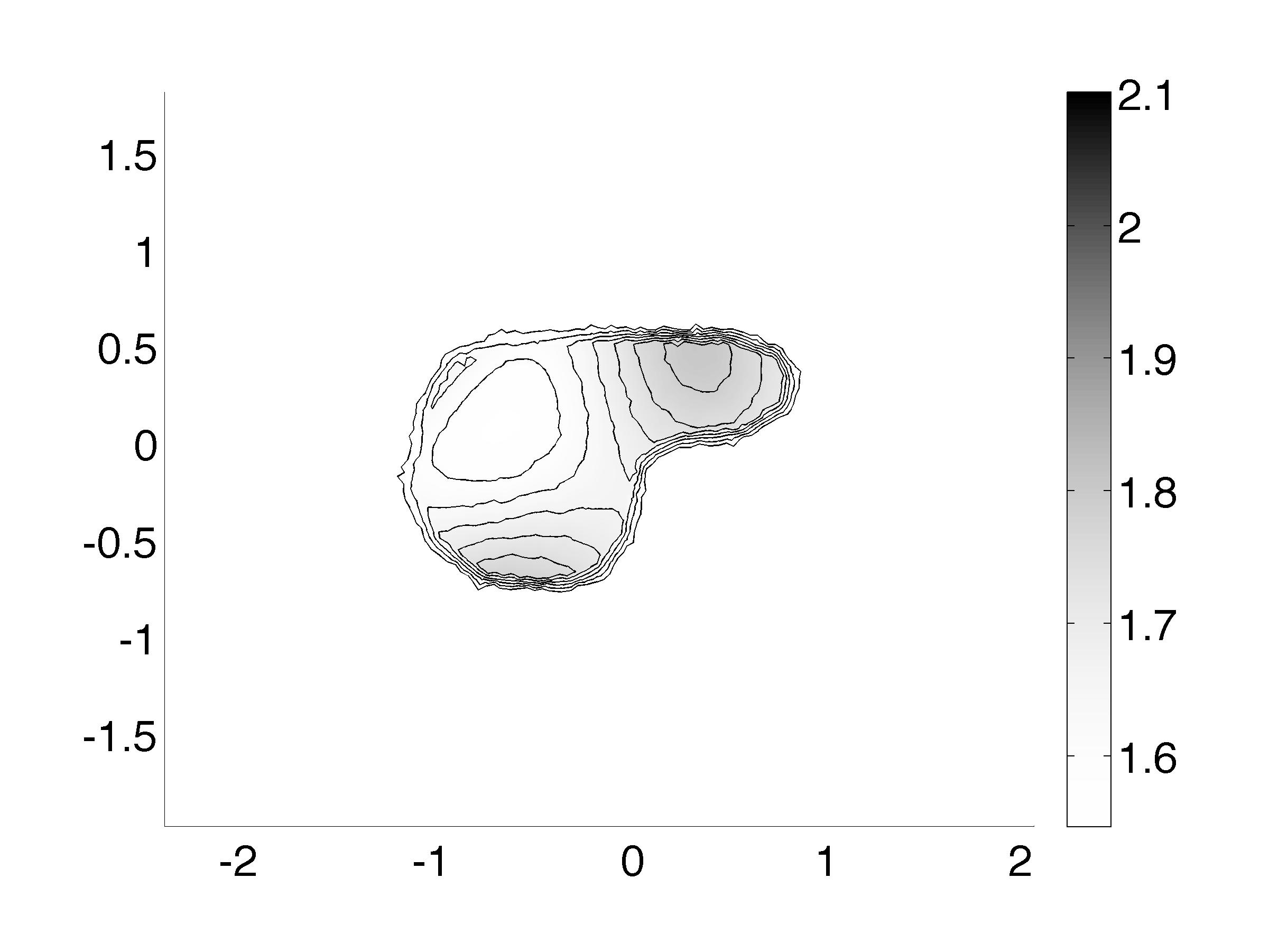}\label{fig:patate.reel}}
\subfloat[Reference index's imaginary part]{\includegraphics[width=0.49\linewidth]{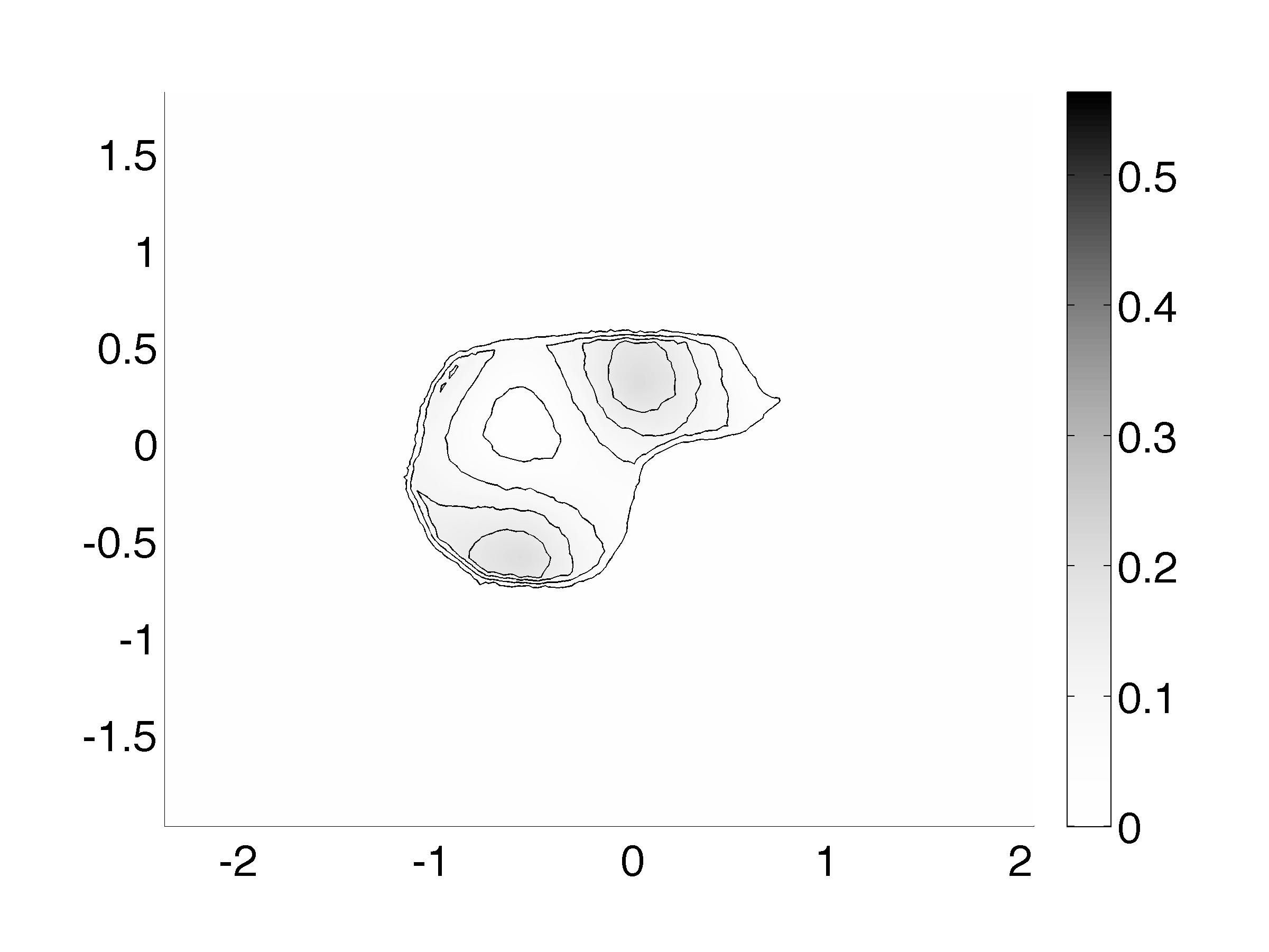}\label{fig:patate.imag}}

\subfloat[Perturbed index's real part]{\includegraphics[width=0.49\linewidth]{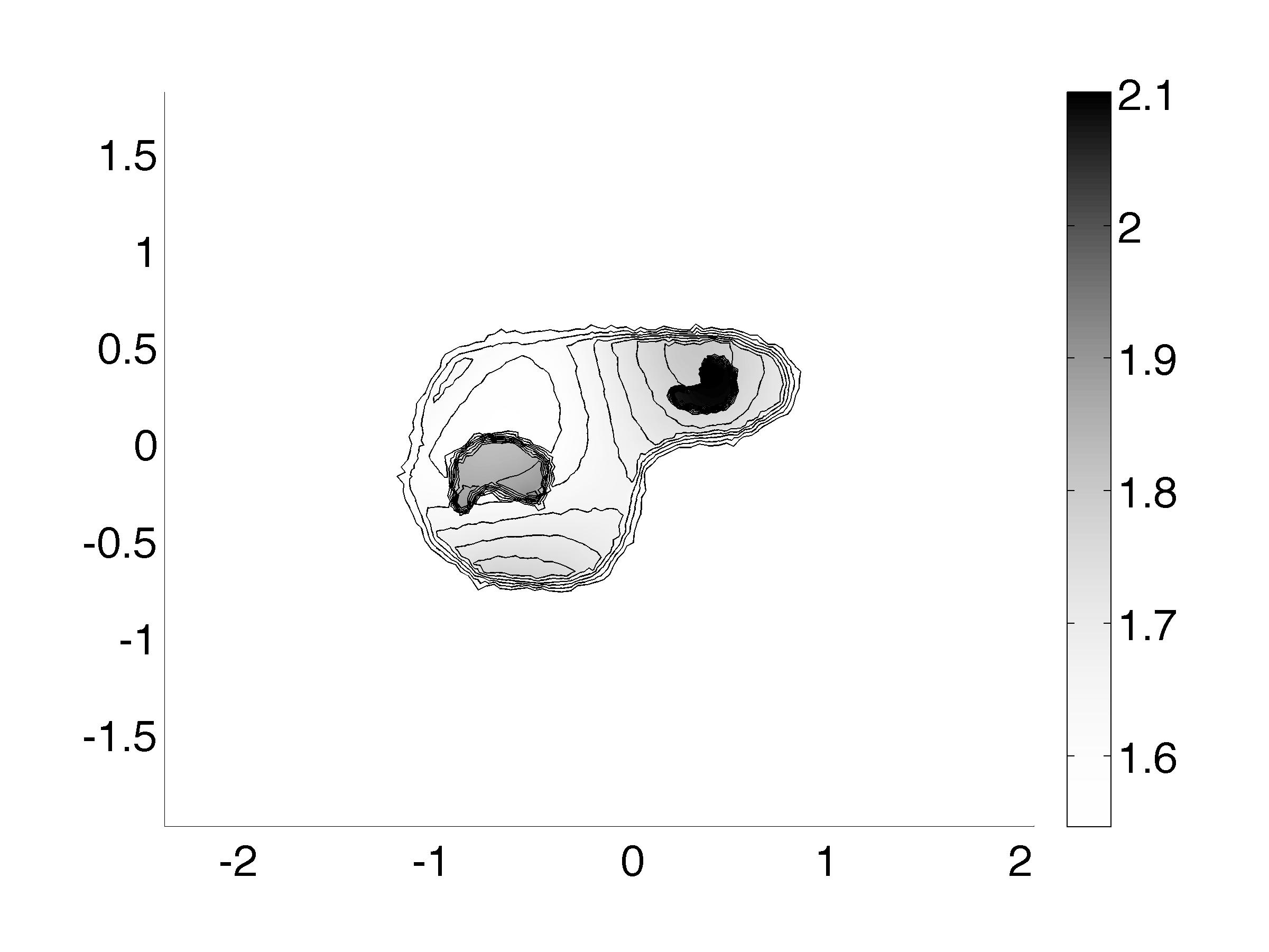}\label{fig:patate.perturbee.reel}}
\subfloat[Perturbed index's imaginary part]{\includegraphics[width=0.49\linewidth]{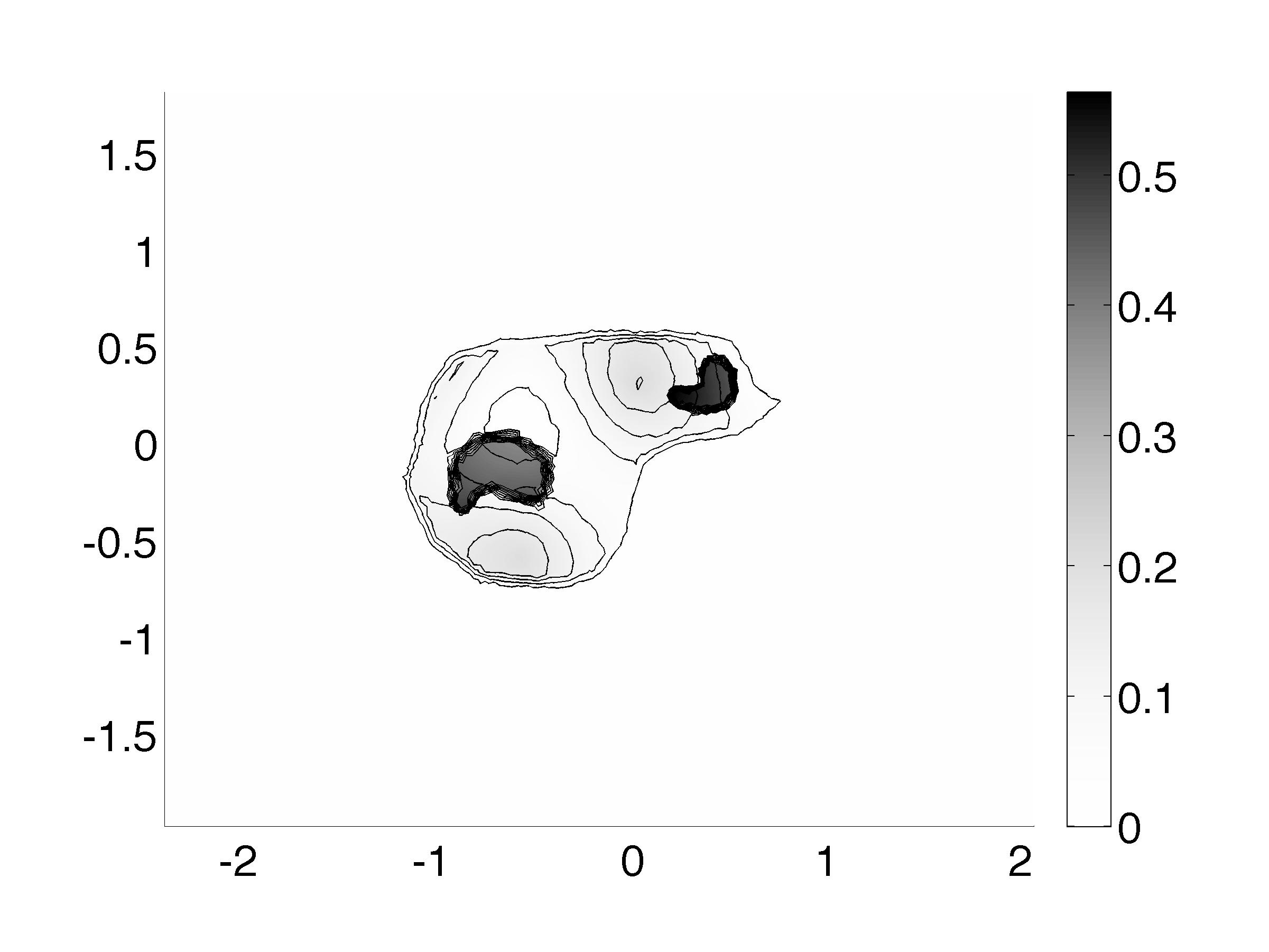}\label{fig:patate.perturbee.imag}}

\subfloat[Values of $\log_{10}\mathcal M_W(z)$ with $2\%$ noise]{\includegraphics[width=0.49\linewidth]{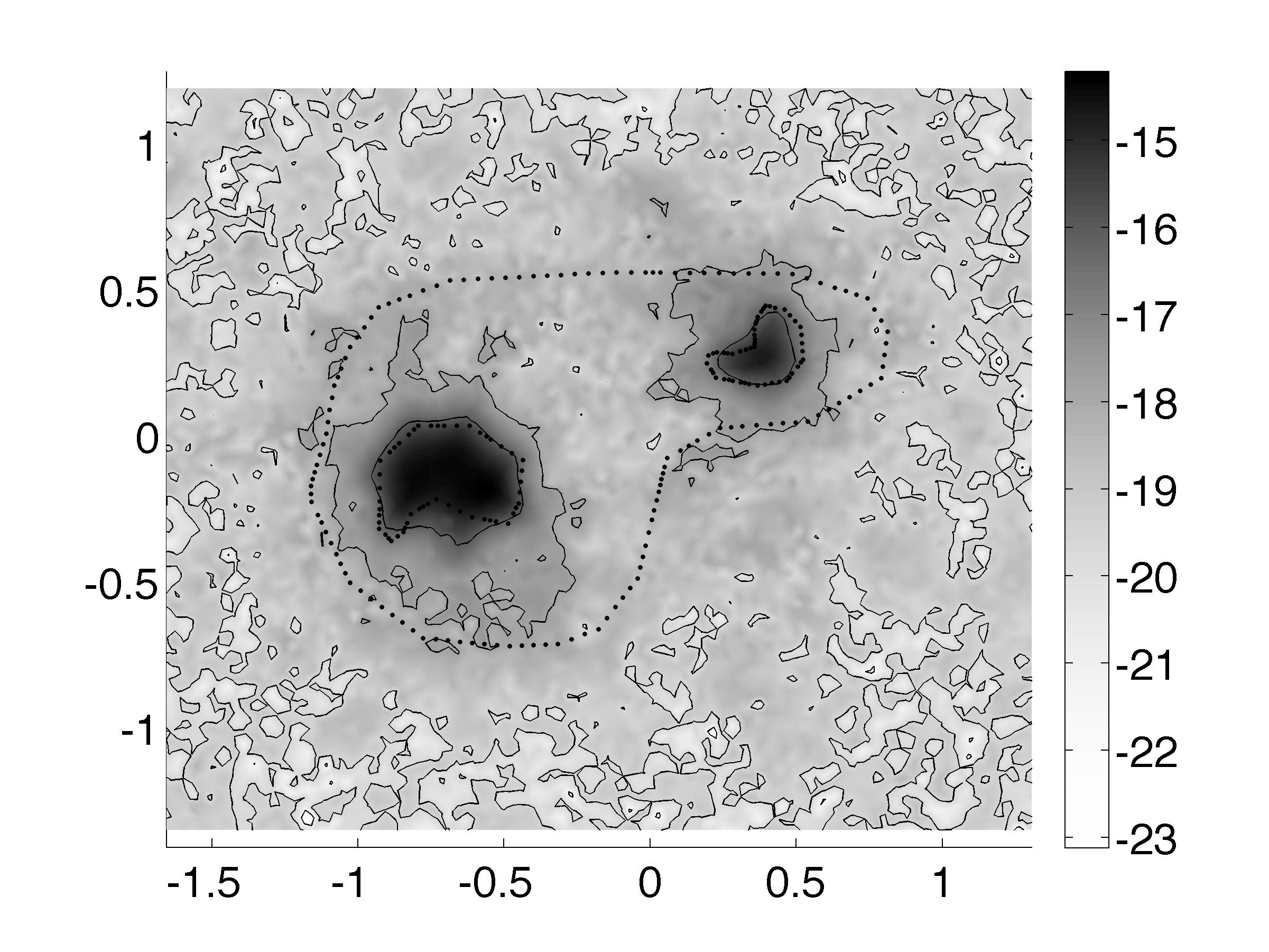}\label{fig:patate.projection}}
\subfloat{
\hspace{-1em}
\raisebox{2em}{
\begin{tabular}[b]{|c|ccc|}
\hline
\textbf{Number of iterations} 
& min & med & max \\
\hline
Steepest descent & 9 & 36 & 198\\
Gradient projection & 9 & 28 & 400\\
Matlab's \texttt{fminunc}  & 6 & 12 & 18\\
Matlab's \texttt{fmincon}  & 8 & 13 & 21\\
\hline
\end{tabular}}}

\caption{Reconstruction of a non-trivial perturbation with the gradient projection algorithm}
\end{figure}

\begin{figure}[htbp]
\centering

\subfloat[Gradient projection algorithm]{\includegraphics[width=0.49\linewidth]{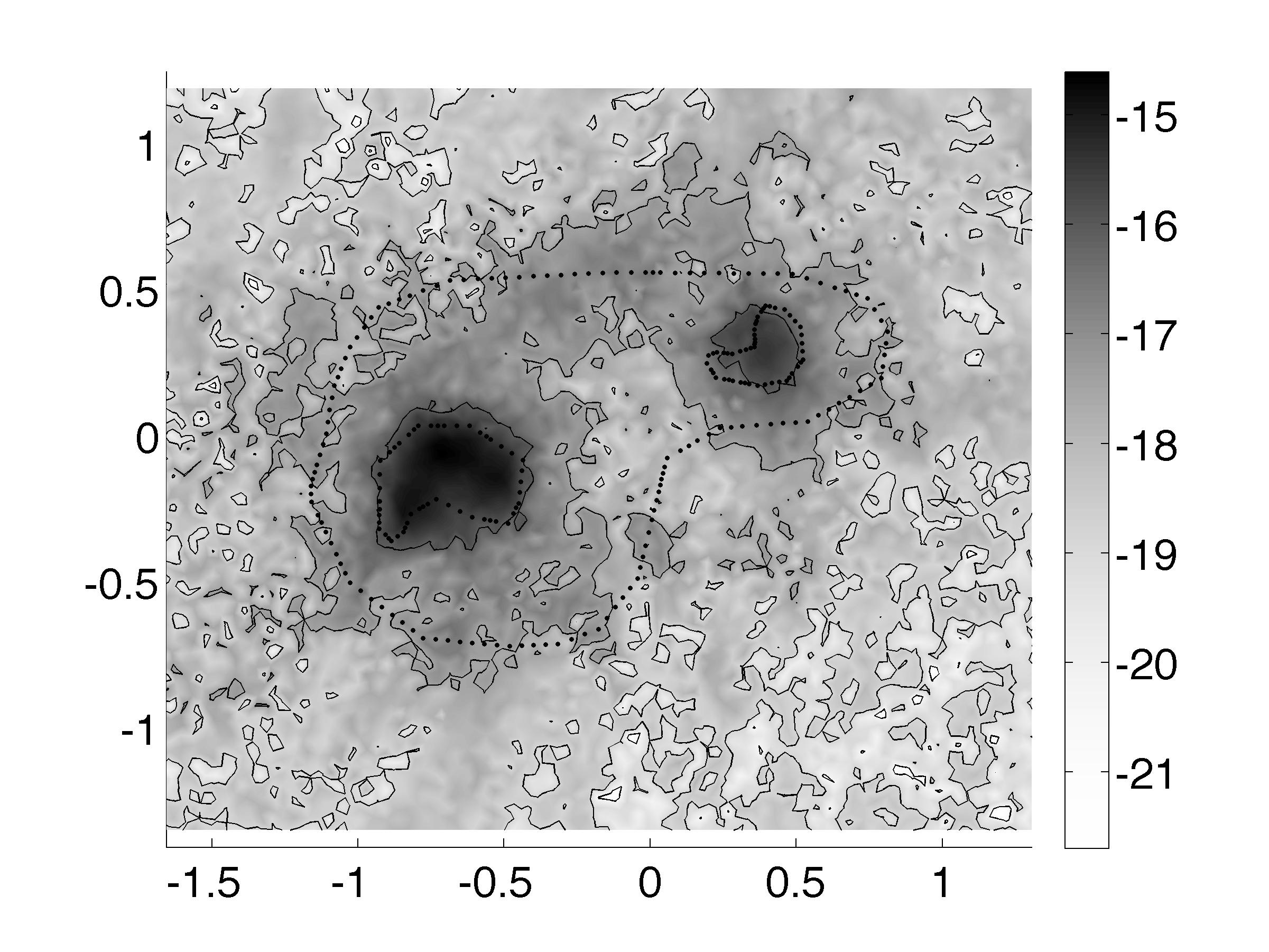}\label{fig:patate.projection.limited}}
\subfloat[Matlab's \texttt{fminunc} function]{\includegraphics[width=0.49\linewidth]{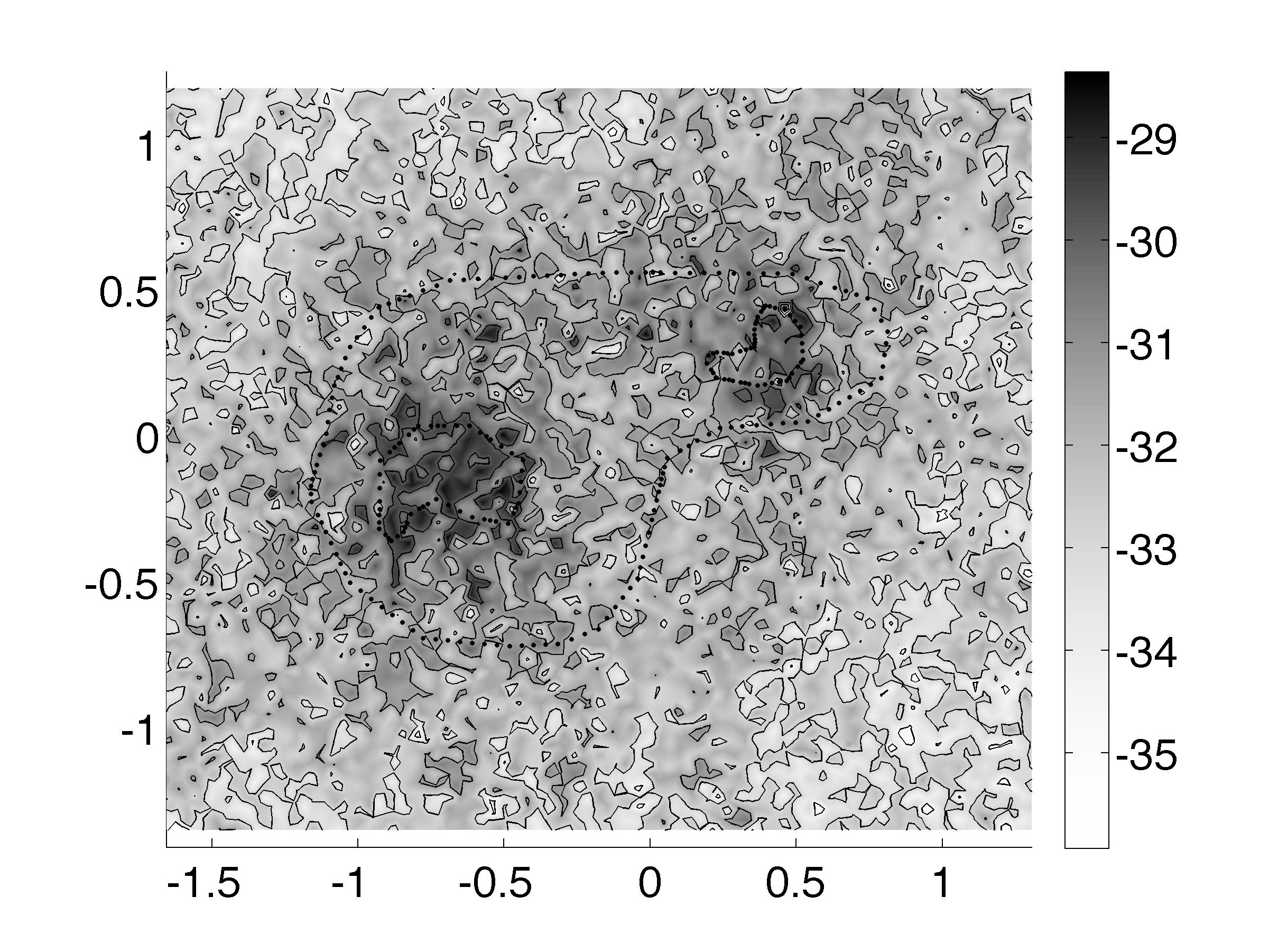}\label{fig:patate.fminunc.limited}}

\medskip 

\begin{tabular}[b]{|c|ccc|}
\hline
\textbf{Number of iterations} 
& min & med & max \\
\hline
Steepest descent & 7 & 41 & 216\\
Gradient projection & 7 & 29 & 158\\
Matlab's \texttt{fminunc}  & 7 & 11 & 17\\
Matlab's \texttt{fmincon}  & 7 & 12 & 20\\
\hline
\end{tabular}

\caption{Reconstruction of a non-trivial perturbation with $2\%$ and limited far-field data}
\end{figure}

It is illustrated in figure~\ref{fig:patate.projection} that comparable results are obtained on a more elaborate example with two non-convex and non-connected defects depicted in figures~\ref{fig:patate.reel}--\ref{fig:patate.perturbee.imag}.

Besides, theorem~\ref{thm:minimisation} is stated under some physical restrictions arising from the use of the scattering operator $(\mathrm{id} + 2ik\abs{\gamma}^2F_{n_0})$.
The numerical methods proposed in this section can however straightforwardly be extended to absorbing media and limited far-field data.
We see in figure~\ref{fig:patate.projection} that defects in complex valued  indices are still correctly  localized, even with $2\%$ uniform random noise added to the measurements.

Finally, even with limited far-field data, the gradient projection algorithm provided a satisfactory reconstruction, displayed in figure~\ref{fig:patate.projection.limited}.
In this last example, the 99 evenly distributed incidence/measurement directions were taken in $[0,\,\frac{4}{3}\pi]$. 
Still, even with this fairly high amount of data, it is to be noted that the four algorithms presented in this paper do not yield comparable results in this case.
Indeed, we see in figure~\ref{fig:patate.fminunc.limited}  that Matlab's functions fail to provide a usable reconstruction, despite our testing on a wide range of optimization parameters.

\section{Conclusion}
We have characterized the localization of defects in an inhomogeneous reference index by an optimization problem that is built only on the available data.
Objective function and feasible set turn out to be very simple.
This problem can thus be solved through a wide range of well known optimization methods, which we have numerically illustrated four examples of.
Yet, some limitations were noticed, regarding the convergence speed on some cases and the required amount of data.
Issues for which successful results were obtained using spectral methods in~\cite{art.grisel.12}.
This opens the perspective of looking for some stabilization, or more robust versions of the proposed optimization algorithms.

\appendix
\subsection{Derivatives of the objective function}\label{sec:gradients}

We give here the derivatives of the functions involved in the computation of $\mathcal M_W(z)$ as defined by~(\ref{eq:I.2}).
However, since the values of the form $f_W$ are real, the differential can not be $\C$-linear.
We therefore have to split the elements of $L^2(S^{d-1})$ into their real and imaginary parts and consider the objective function on pairs of real-valued functions to obtain proper $\R$-linear differentials.
The induced gradient is then given in the following lemma.

\begin{lemma}\label{lem:df}
The gradient of the form $\mathbf{f_{W}^4} : L^2(S^{d-1},\R) \times L^2(S^{d-1},\R) \to \R$, defined by
\begin{equation*}%\label{def:f}
    \mathbf{f_{W}^4}\mat{\phi \\ \psi} \pardef \big(f_{W}(\Psi)\big)^4,\quad \Psi \pardef \phi+i\psi,
\end{equation*}
is given by
\begin{equation}\label{eq:gradf4}
    \grad \mathbf{f_{W}^4}\mat{\phi \\ \psi} \pardef 4
    %\mat{\re W_R\Psi\ps{ W_R \Psi}{\Psi}+\re W_I\Psi\ps{ W_I \Psi}{\Psi} \\ \im W_R\Psi\ps{ W_R \Psi}{\Psi}+\im W_I\Psi\ps{ W_I \Psi}{\Psi}}
    \mat{\re Gf\mat{\phi \\ \psi} \\ \im Gf\mat{\phi \\ \psi}},
\end{equation}
    where the function $Gf$ is an endomorphism on $(L^2(S^{d-1},\R))^2$ defined, with help of the self-adjoint parts $W_R \pardef ( W +  W\etoile)/2$ and $W_I \pardef ( W -  W\etoile)/2i$, by
    \begin{equation}\label{eq:Gf}
    	G_f\mat{\phi \\ \psi} \pardef \ps{ W_R \Psi}{\Psi}W_R\Psi + \ps{ W_I \Psi}{\Psi}W_I\Psi,\quad \Psi \pardef \phi+i\psi.
    \end{equation}

    Also, the second derivative of $\mathbf{f_{W}^4}$ is defined on $(L^2(S^{d-1},\R))^4$, for each point $(\phi,\psi) \in  (L^2(S^{d-1},\R))^2$, by
    \begin{equation}\label{eq:d2f4}
	\left[D^2\mathbf{f_{W}^4}\mat{\phi \\ \psi}\right](\mat{a \\ b},\mat{c \\ d}) = 4\ps{\mat{\re \left[H_f\mat{\phi \\ \psi}\right]\mat{a \\ b} \\ \im  \left[H_f\mat{\phi \\ \psi}\right]\mat{a \\ b}}}{\mat{c \\ d}},
    \end{equation}
    where the operator $\left[H_f\mat{\phi \\ \psi}\right]$ is an endomorphism on \(L^2(S^{d-1})\) defined for each $\mat{\phi \\ \psi} \in (L^2(S^{d-1}))^2$ by
    \begin{equation*}
    \left[H_f\mat{\phi \\ \psi}\right]\mat{a \\ b} \pardef 2\re\ps{W_R\Psi}{\Phi}W_R\Psi + \ps{W_R\Psi}{\Psi}W_R\Phi + 2\re\ps{W_I\Psi}{\Phi}W_I\Psi + \ps{W_I\Psi}{\Psi}W_I\Phi,
    \end{equation*}
    where $\Psi \pardef \phi+i\psi$ and $\Phi \pardef a+ib$.
\end{lemma}

\begin{proof}
First, for $\mat{\phi \\ \psi}$ and $\mat{u \\ v}$ in $L^2(S^{d-1},\R) \times L^2(S^{d-1},\R)$ we denote
\begin{align*}
	 \Psi \pardef \phi+i\psi,\\
	 \Phi \pardef u+iv.
\end{align*}
   Let then \(\mathbf{f_{W}^2} : L^2(S^{d-1},\R) \times L^2(S^{d-1},\R) \to \C\) be defined by  
   \[\mathbf{f_{W}^2}\mat{\phi \\ \psi} \pardef \ps{ W\Psi}{\Psi}_{L^2(S^{d-1})},\]
    so we have \(\mathbf{f_{W}^4}\mat{\phi \\ \psi}=\abs{\mathbf{f_{W}^2}\mat{\phi \\ \psi}}^2\). %and thus
%    \begin{equation}\label{eq:df4}
%        D\mathbf{f_{W}^4}\mat{\phi \\ \psi}\mat{u \\ v} = 2\re \ps{\mathbf{f_{W}^2}\mat{\phi \\ \psi}}{D\mathbf{f_{W}^2}\mat{\phi \\ \psi}\mat{u \\ v}}.
%    \end{equation}
   Moreover, the operator $W$ is not self-adjoint but is nevertheless an endomorphism.
   Hence, we can use its real and imaginary parts as defined by
    \[ W_R \pardef ( W +  W\etoile)/2,\quad  W_I \pardef ( W -  W\etoile)/2i.\]
    It follows that
    $%\begin{align*}
        \ps{ W_R \Psi}{\Psi}
%            &=\ps{ W \Psi}{\Psi}/2 + \ps{ W\etoile \Psi}{\Psi}/2\\
%            &=\ps{ W \Psi}{\Psi}/2 + \ps{\Psi}{ W\Psi}/2\\
%            &
            =\re \ps{ W \Psi}{\Psi}
            $
%.\end{align*}
%    In the same way \[
and $
\ps{ W_I \Psi}{\Psi} = \im \ps{ W \Psi}{\Psi}
$%
. %\]
%    As a consequence, it comes \[\mathbf{f_{W}^2}\mat{\phi \\ \psi} = \ps{ W_R \Psi}{\Psi} + i\ps{ W_I \Psi}{\Psi},\] and 
Since the operators \( W_R\) et \( W_I\) are self-adjoint,
%    from this we obtain that, 
with %\(\Psi \pardef \phi+i\psi\) and 
\(\Phi \pardef u+iv\) we obtain
%    \[Dg\mat{\phi \\ \psi}\mat{u \\ v} = \ps{ W_R \Psi}{\Phi} + \ps{ W_R \Phi}{\Psi} + i\ps{ W_I \Psi}{\Phi} + i\ps{ W_I \Phi}{\Psi}.\]
%    Since the operators \( W_R\) et \( W_I\) are self-adjoint, this can be written
    \begin{align*}
        D\mathbf{f_{W}^2}\mat{\phi \\ \psi}\mat{u \\ v} 
 %           &=\ps{ W_R \Psi}{\Phi} + \ps{ \Phi}{ W_R\etoile \Psi} + i\ps{ W_I \Psi}{\Phi} + i\ps{ \Phi}{ W_I\etoile \Psi}\nonumber\\
%            &=\ps{ W_R \Psi}{\Phi} + \overline{\ps{ W_R \Psi}{ \Phi}} + i\ps{ W_I \Psi}{\Phi} + i\overline{\ps{ W_I \Psi}{ \Phi}}\nonumber\\
            &= 2\re \ps{ W_R \Psi}{\Phi}  + 2i \re \ps{ W_I \Psi}{\Phi}.%\label{eq:df42}
    \end{align*}
    
Hence, the differential of $\mathbf{f_{W}^4}$ is given by
\begin{align*}
%\lefteqn{
D\mathbf{f_{W}^4}\mat{\phi \\ \psi}\mat{u \\ v}%} \\
&=2\re \ps{\ps{ W_R \Psi}{\Psi} + i\ps{ W_I \Psi}{\Psi}}{2\re \ps{ W_R \Psi}{\Phi} + 2i \re \ps{ W_I \Psi}{\Phi}}\\
%&=4\ps{\mat{\re \ps{ W_R \Psi}{\Phi} \\ \re \ps{ W_I \Psi}{\Phi}}}{\mat{\ps{ W_R \Psi}{\Psi} \\ \ps{ W_I \Psi}{\Psi}}}\\
&=4\re\ps{\ps{ W_R \Psi}{\Psi}W_R\Psi + \ps{ W_I \Psi}{\Psi}W_I\Psi}{\Phi}\\
%&=4\ps{\mat{\ps{\re\Phi}{\re W_R \Psi}+\ps{\im\Phi}{\im W_R \Psi} \\ \ps{\re\Phi}{\re W_I \Psi}+\ps{\im\Phi}{\im W_I \Psi}}}{\mat{\ps{ W_R \Psi}{\Psi} \\ \ps{ W_I \Psi}{\Psi}}}\\
%&=4\ps{\mat{\re\Phi \\ \im\Phi}}{\mat{\re W_R\Psi\ps{ W_R \Psi}{\Psi}+\re W_I\Psi\ps{ W_I \Psi}{\Psi} \\ \im W_R\Psi\ps{ W_R \Psi}{\Psi}+\im W_I\Psi\ps{ W_I \Psi}{\Psi}}}
&=4\re \ps{Gf(\Psi)}{\Phi}
,\end{align*}
where $G_f$ is defined by~(\ref{eq:Gf}).
%    \begin{equation*}%\label{eq:Gf}
%    	G_f(\Psi) \pardef \ps{ W_R \Psi}{\Psi}W_R\Psi + \ps{ W_I \Psi}{\Psi}W_I\Psi.
%    \end{equation*}
The gradient~(\ref{eq:gradf4}) is then written by recalling that for hermitian inner products, it holds that
\begin{equation}\label{eq:ps.reel}
\re \ps{f}{g} = \ps{\mat{\re f \\ \im f}}{\mat{\re g \\ \im g}}.
\end{equation}

Finally, we get the second derivative by differentiating the gradient.
Since the operators $W_R$ and $W_I$ are $\C$-linear and self-adjoint, it comes
\begin{equation*}
	DG_f(\Psi)(\Phi) = 2\re\ps{W_R\Psi}{\Phi}W_R\Psi + \ps{W_R\Psi}{\Psi}W_R\Phi + 2\re\ps{W_I\Psi}{\Phi}W_I\Psi + \ps{W_I\Psi}{\Psi}W_I\Phi.
\end{equation*}
\qed\end{proof}

As a consequence, we also have to adapt the projection $P_{\mathcal{C}_z}$~(\ref{eq:proj}) and use its counterpart $\mathbf{P_{\mathcal{C}_z}}$ defined as an affine endomorphism of  \(L^2(S^{d-1},\R) \times L^2(S^{d-1},\R)\) by
\begin{equation}\label{eq:PCz}
	\mathbf{P_{\mathcal{C}_z}}\mat{\phi \\ \psi} \pardef \mat{\re P_{\mathcal{C}_z}\Psi \\ \im P_{\mathcal{C}_z}\Psi},\quad \Psi = \phi+i\psi.
\end{equation}
%
%\[\mathbf{P_{\mathcal{C}_z}}\mat{\phi \\ \psi} \pardef  \mathbf{\overset{\to}{P}_{\mathcal{C}_z}}\mat{\phi \\ \psi}+ \frac{1}{\norme{u_{n_0}(\cdot,z)}^2}\mat{\re \overline{u_{n_0}(\cdot,z)} \\ \im \overline{u_{n_0}(\cdot,z)}} ,\]
%where the linear part $\mathbf{\overset{\to}{P}_{\mathcal{C}_z}}$ of the projection is given by
%\[\mathbf{\overset{\to}{P}_{\mathcal{C}_z}} \pardef \mat{P_{\mathcal{C}_z} \\ }.\]
%
%\[\Psi-\frac{\ps{\Psi}{\overline{u_{n_0}(\cdot,z)}}}{\norme{u_{n_0}(\cdot,z)}^2}\overline{u_{n_0}(\cdot,z)}\]
The gradient  of the objective function's projection, denoted by $\mathbf{P_{f_{W}^4}}$, is then given by
\begin{equation*}%
    \grad \mathbf{P_{f_{W}^4}} \mat{\phi \\ \psi} = \mat{
    \re  {(\overset{\to}{P}_{\mathcal{C}_z})^\star} Gf ( {P_{\mathcal{C}_z}}\Psi) \\ 
    \im {(\overset{\to}{P}_{\mathcal{C}_z})^\star} Gf ( {P_{\mathcal{C}_z}}\Psi)
    },\quad \Psi = \phi+i\psi,
\end{equation*}
where ${\overset{\to}{P}_{\mathcal{C}_z}}$ is the linear part of the projection and where $G_f$ is given by~(\ref{eq:Gf}).

\subsection{Finite dimension approximation}\label{sec:dimension.finie}

For the finite dimension approximation, the operators  ${W_R}$ and ${W_I}$ have a complex matrix representation. % such that the approximation of $\ps{ W\Psi}{\Psi}_{L^2(S^{d-1})}$ can be written
%\[\ps{ W\Psi}{\Psi}_{L^2(S^{d-1})} \approx x\etoile W_R x + i x\etoile W_I x.\]
In order to write the gradient in more natural terms of matrix-vector products, we thus denote $\mathbf{W_R}$ and $\mathbf{W_I}$ the corresponding real valued expanded matrices defined by
\begin{align*}
%	\mathbf{W} \pardef \mat{\re({W}) & -\im({W}) \\ \im({W}) & \re({W})},\
	\mathbf{W_R} \pardef \mat{\re({W_R}) & -\im({W_R}) \\ \im({W_R}) & \re({W_R})},\quad
	\mathbf{W_I} \pardef \mat{\re({W_I}) & -\im({W_I}) \\ \im({W_I}) & \re({W_I})}.
\end{align*}
Moreover, we assume the standard change of basis $M^\frac{1}{2}$, where $M$ is a discretization of the inner product, so that the transposition correctly represents the adjoint of operators.
Furthermore, with $\phi$ and $\psi$ two elements  of the discretized version of $L^2(S^{d-1})$, in the basis $M^\frac{1}{2}$, denote \[\mathbf{x} = \mat{\phi \\ \psi}.\]
With these notations it comes that
\begin{equation*}
	\mathbf{W_R} \mathbf{x} = \mat{\re W_R(\phi+i\psi) \\ \im W_R(\phi+i\psi)},\quad \mathbf{W_I} \mathbf{x} = \mat{\re W_I(\phi+i\psi) \\ \im W_I(\phi+i\psi)},
\end{equation*}
and by recalling~(\ref{eq:ps.reel}), we have
\begin{equation*}\begin{array}{lll}
	\mathbf{x}^T \mathbf{W_R} \mathbf{x} &= \re \ps{{W_R}(\phi+i\psi)}{(\phi+i\psi)} 
	&= \ps{{W_R}(\phi+i\psi)}{(\phi+i\psi)} ,\\
	\mathbf{x}^T \mathbf{W_I} \mathbf{x}  &= \re \ps{{W_I}(\phi+i\psi)}
	{(\phi+i\psi)} &= \ps{{W_I}(\phi+i\psi)}{(\phi+i\psi)}.
\end{array}\end{equation*}
%where $\mathbf{O}$ is the expanded matrix of the multiplication by the complex value $-i$ given by \[\mathbf{O} \pardef \mat{0 & Id \\ -Id & 0}.\]

It follows from~(\ref{eq:gradf4}) that  the (expanded) gradient of the form $f^4_W$ is given by
\[ \grad \mathbf{f_{W}^4} (\mathbf{x}) = 4 \mathbf{W_R} \mathbf{x} (\mathbf{x}^T\mathbf{W_R} \mathbf{x}) + 4 \mathbf{W_I} \mathbf{x}(\mathbf{x}^T \mathbf{W_I} \mathbf{x}).\]
%where $\mathbf M$ is the mass matrix discretizing the expanded inner product.
It also follows from~(\ref{eq:d2f4}) that the corresponding hessian matrix $\mathbf{H}_{f^4_W}(\mathbf{x})$  is given~by
\begin{equation*}
    \mathbf{H}_{f^4_W} (\mathbf{x})
    = 8 \mathbf{W_R} \mathbf{x} \otimes \mathbf{W_R} \mathbf{x}
    + 4 \mathbf{W_R} (\mathbf{x}^T \mathbf{W_R} \mathbf{x}) 
    +8 \mathbf{W_I} \mathbf{x} \otimes \mathbf{W_I} \mathbf{x}
    + 4 \mathbf{W_I} (\mathbf{x}^T \mathbf{W_I} \mathbf{x}),
\end{equation*}
where the tensor product between two column vectors $\mathbf{a} \pardef \mat{a_1 \\ a_2 \\ \vdots}$ and $\mathbf{b} \pardef \mat{b_1 \\ b_2 \\ \vdots}$ is defined by the matrix given in columns by
\[\mathbf{a} \otimes \mathbf{b} \pardef \mat{\mathbf{a}b_1 & \mathbf{a}b_2 \dots}.\]

Finally, it then comes from a straightforward calculation that the finite dimension approximations of the gradient and the hessian matrix for the form $\mathbf{P_{f_{W}^4}}$ are 
\[\grad \mathbf{P_{f_{W}^4}} (\mathbf{x}) = \mathbf{\overset{\to}{P}}^T \grad \mathbf{f_W^4}(\mathbf{P}_{\mathcal C_z}\mathbf{x}),\]
and
\begin{equation*}
    \mathbf{H}_{P_{f^4_W}} (\mathbf{x})
    =\mathbf{\overset{\to}{P}}^T  
       \mathbf{H}_{f^4_W} (\mathbf{P}_{\mathcal C_z}\mathbf{x})
       \mathbf{\overset{\to}{P}},
\end{equation*}
where $\mathbf{\overset{\to}{P}}$ is the expanded matrix representation of the projection's linear part.% and $\mathbf{\overset{\to}{P}}\etoile$ is the representation of its adjoint, with respect to the scalar product in \(\big(L^2(S^{d-1},\R)\big)^2\).

%\begin{acknowledgements}
%
%Support for some of the authors of this work was provided by the FRAE (Fondation de Recherche pour l'A\'eronautique et l'Espace, \texttt{http://www.fnrae.org/}), research project IPPON.
%\end{acknowledgements}

\bibliography{yann-biblio}   % name your BibTeX data base

\begin{thebibliography}{10}

\bibitem{art.arens.04}
T.~Arens.
\newblock {Why linear sampling works}.
\newblock {\em Inverse Problems}, 20:163, 2004.

\bibitem{art.coleman.96}
Thomas~F. Coleman and Yuying Li.
\newblock An interior trust region approach for nonlinear minimization subject
  to bounds.
\newblock {\em SIAM J. Optim.}, 6(2):418--445, 1996.

\bibitem{art.collino.2002}
F.~Collino, M'B. Fares, and H.~Haddar.
\newblock On the validation of the linear sampling method in electromagnetic
  inverse scattering problems.
\newblock Research Report 4665, INRIA, 2002.

\bibitem{art.colton.03}
D.~Colton.
\newblock Inverse acoustic and electromagnetic scattering theory.
\newblock In Gunther Uhlman, editor, {\em Inside out: inverse problems and
  applications}, volume~47 of {\em Math. Sci. Res. Inst. Publ.}, pages 67--110.
  Cambridge Univ. Press, Cambridge, 2003.

\bibitem{art.colton.96}
D.~Colton and A.~Kirsch.
\newblock A simple method for solving inverse scattering problems in the
  resonance region.
\newblock {\em Inverse Problems}, 12(4):383--393, 1996.

\bibitem{book.colton.1}
D.~Colton and R.~Kress.
\newblock {\em Inverse acoustic and electromagnetic scattering theory},
  volume~93 of {\em Applied Mathematical Sciences}.
\newblock Springer-Verlag, Berlin, second edition, 1998.

\bibitem{art.colton.97}
D.~Colton, M.~Piana, and R.~Potthast.
\newblock A simple method using {M}orozov's discrepancy principle for solving
  inverse scattering problems.
\newblock {\em Inverse Problems}, 13(6):1477--1493, 1997.

\bibitem{art.grisel.12}
Y.~Grisel, V.~Mouysset, P-A. Mazet, and J-P. Raymond.
\newblock Determining the shape of defects in non-absorbing inhomogeneous media
  from far-field measurements.
\newblock {\em Inverse Problems}, 28:055003, 2012.

\bibitem{art.kirsch.98}
A.~Kirsch.
\newblock Characterization of the shape of a scattering obstacle using the
  spectral data of the far field operator.
\newblock {\em Inverse Problems}, 14(6):1489--1512, 1998.

\bibitem{art.kirsch.99}
A.~Kirsch.
\newblock {Factorization of the far-field operator for the inhomogeneous medium
  case and an application in inverse scattering theory}.
\newblock {\em Inverse Problems}, 15:413--429, 1999.

\bibitem{art.kirsch.00}
A.~Kirsch.
\newblock New characterizations of solutions in inverse scattering theory.
\newblock {\em Applicable Analysis}, 76:319--350, 2000.

\bibitem{art.kirsch.02}
A.~Kirsch.
\newblock The {MUSIC} algorithm and the factorization method in inverse
  scattering theory for inhomogeneous media.
\newblock {\em Inverse Problems}, 18(4):1025--1040, 2002.

\bibitem{book.kirsch.08}
A.~Kirsch and N.I. Grinberg.
\newblock {\em The factorization method for inverse problems}, volume~36 of
  {\em Oxford Lecture Series in Mathematics and its Applications}.
\newblock Oxford University Press, Oxford, 2008.

\bibitem{art.nachman.07}
A.~I. Nachman, L.~P{\"a}iv{\"a}rinta, and A.~Teiril{\"a}.
\newblock On imaging obstacles inside inhomogeneous media.
\newblock {\em J. Funct. Anal.}, 252(2):490--516, 2007.

\bibitem{art.potthast.06}
R.~Potthast.
\newblock A survey on sampling and probe methods for inverse problems.
\newblock {\em Inverse Problems}, 22(2):R1--R47, 2006.

\bibitem{art.rosen.60}
J.B. Rosen.
\newblock The gradient projection method for nonlinear programming. part i.
  linear constraints.
\newblock {\em Journal of the Society for Industrial and Applied Mathematics},
  8(1):181--217, 1960.

\bibitem{conf.venkov.08}
G.~Venkov.
\newblock Atkinson-{W}ilcox expansion theorem for inhomogeneous media.
\newblock In {\em Math. Proc. R. Ir. Acad.}, volume 108, pages 19--25, 2008.

\end{thebibliography}
\end{document}